\documentclass[12pt]{amsart}

\usepackage{xcolor}
\usepackage[pagebackref]{hyperref}
\hypersetup{
   colorlinks,
    linkcolor={red!60!black},
    citecolor={blue!60!black},
    urlcolor={blue!90!black}
}

\usepackage{amsmath}
\usepackage{pdfsync}         %
\usepackage[active]{srcltx}  %

\usepackage{enumitem}
\usepackage{comment}
\usepackage{amssymb}
\usepackage{cancel}
\usepackage[T1]{fontenc}
\usepackage[utf8]{inputenc}   %
\usepackage{lmodern} 
\usepackage{cleveref}
\usepackage{mathtools}

\usepackage{graphicx}

\usepackage{tikz}
\usetikzlibrary{patterns,arrows.meta}
\definecolor{bandteal}{RGB}{104,170,160}
\definecolor{gridred}{RGB}{204,80,60}

\usepackage{palatino}
\usepackage[left=3cm,top=3.8cm,right=3cm]{geometry}        %

\usepackage{bm} %

\numberwithin{equation}{section}

\DeclareMathOperator{\diam}{diam}

\DeclareMathOperator{\proj}{proj}

\DeclareMathOperator{\Err}{Err}

\newcommand{\Fav}{\operatorname{Fav}}
\newcommand{\Var}{\operatorname{Var}}

\newcommand{\E}{\mathbb{E}}

\newcommand{\N}{\mathbb{N}}
\renewcommand{\P}{\mathbb{P}}
\newcommand{\R}{\mathbb{R}}

\newcommand{\vv}{{\bf v}}

\newcommand{\cA}{\mathcal{A}}
\newcommand{\cB}{\mathcal{B}}

\newcommand{\cD}{\mathcal{D}}

\newcommand{\cG}{\mathcal{G}}
\newcommand{\cH}{\mathcal{H}}
\newcommand{\cI}{\mathcal{I}}
\newcommand{\cL}{\mathcal{L}}
\newcommand{\cP}{\mathcal{P}}

\newcommand{\cS}{\mathcal{S}}

\newcommand{\X}{\mathfrak X}

\newcommand{\LEI}{\mathsf{LEI}}
\newcommand{\rvL}{\mathbf L}
\newcommand{\Count}{\mathsf{Count}}

\usepackage{mathrsfs}

\newcommand{\emp}{\varnothing}

\renewcommand\epsilon{{\varepsilon}}
\newcommand\eps{{\varepsilon}}

\usepackage{dsfont}

\newcommand{\bbone}{\mathds{1}}

\def\XXint#1#2#3{{\setbox0=\hbox{$#1{#2#3}{\int}$ }
\vcenter{\hbox{$#2#3$ }}\kern-.6\wd0}}

\usepackage{thmtools}
\declaretheorem[numberwithin=section,name=Theorem]{theorem}
\declaretheorem[sibling=theorem,name=Lemma]{lemma}
\declaretheorem[sibling=theorem,name=Proposition]{proposition}
\declaretheorem[sibling=theorem,name=Corollary]{corollary}

\declaretheorem[sibling=theorem,name=Definition,style=definition]{definition}
\declaretheorem[sibling=theorem,name=Remark,style=definition]{remark}

\declaretheorem[sibling=theorem,name=Example,style=remark]{example}

\definecolor{alanorange}{RGB}{204, 112, 0} %

\title{Sharp Favard length of random Cantor sets}

\author{Alan Chang}
\address{AC: Department of Mathematics, Washington University in St.\ Louis. 1 Brookings Drive, St.\ Louis, MO 63105, USA}
\email{alanchang@math.wustl.edu}

\author{Pablo Shmerkin}
\address{PS: Department of Mathematics, the University of British Columbia. 1984 Mathematics Road, Vancouver BC V6T 1Z2, Canada}
\email{pshmerkin@math.ubc.ca}

\author{Ville Suomala}
\address{VS: Mathematical Sciences, P.O.Box 8000, FI-90014, University of
Oulu, Finland}
\email{ville.suomala@oulu.fi}

\thanks{A.C. was supported by NSF grant DMS-2247233. P.S. was partially supported by an NSERC Discovery Grant and by a Simons Fellowship. V.S. was supported by the Research Council of Finland via the project “\emph{Fractals and randomness}”, grant no. 368817. This project began in June 2024 at the Banff International Research Station while we participated the workshop \emph{On the Interface of Geometric Measure Theory and Harmonic Analysis}.}

\begin{document}

\begin{abstract}
     We show that for a large class of planar $1$-dimensional random fractals $S$, the Favard length $\Fav(S(r))$ of the neighborhood $S(r)$ is comparable to $\log^{-1}(1/r)$, matching a universal lower bound; up to now, this was only known in expectation for a few concrete models. In particular, we show that there exist $1$-Ahlfors regular sets with the fastest possible Favard length decay.  For a wide class of planar one-dimensional ``grid random fractals'', including fractal percolation and its Ahlfors-regular variants, we further show that $\Fav(S(r))/\log(1/r)$ converges almost surely, and we identify the limit explicitly. Furthermore, we prove that for some $1$-dimensional Ahlfors-regular random fractals $S$, the Favard length of $S(r)$ decays instead like $\log\log(1/r)/\log(1/r)$, showing that the $1/\log(1/r)$ decay is not universal among random fractals, as might be expected from previous results.
\end{abstract}

\maketitle

\setcounter{tocdepth}{3} %
\tableofcontents

\section{Introduction}

\subsection{The minimal Favard length problem}

If $S$ is a Borel subset of $\R^2$, the \emph{Favard length} of $S$ is defined as
\[
    \Fav(S) = \int_{0}^{\pi} |\proj_{\theta} S| \,d\theta,
\]
where $\proj_{\theta}$ is orthogonal projection in direction $\theta$, and $|\cdot|$ denotes one-dimensional Lebesgue measure.

The Favard length is a key quantity in geometric measure theory. This is exemplified by the classical Besicovitch projection theorem (see \cite[Chapter 18]{Mattila95} for a modern exposition): a $1$-set $S\subset\R^2$ (that is, a Borel set with $0<\mathcal{H}^1(S)<\infty$) has zero Favard length if and only if it is purely unrectifiable, that is, if and only if its intersection with all Lipschitz curves has zero length. Motivated by this, there is an ongoing program of quantifying the Besicovitch projection theorem in various ways. In a recent breakthrough, D.~D\k{a}browski \cite{Dabrowski24} proved that if a 1-Ahlfors-regular set has ``large'' Favard length, then it contains a ``large'' piece of a Lipschitz graph. The Ahlfors regularity assumption is natural and often made in this setting as there are counterexamples otherwise.

Let $S(r)$ denote the $r$-neighborhood of the set $S$.
A related line of research concerns studying the decay of $\Fav(S(r))$ as $r\downarrow 0$ for \emph{specific} purely unrectifiable sets. As shown by P.~Mattila \cite[p.86]{Mat90} (combined with Frostman's Lemma), one has
\begin{equation}\label{Fav_lb_deterministic}
\Fav(S(r)) \gtrsim \frac{\mathcal{H}_\infty^1(S)}{\log(1/r)}\,, 
\end{equation}
where $\mathcal{H}_\infty^1$ denotes $1$-dimensional Hausdorff content, and the implicit constant is universal. This suggests the problem of whether there exist $1$-sets $S$ such that 
\begin{align} \label{eq:minimal-fav}
\Fav(S(r)) \sim \frac{1}{\log(1/r)} \quad\text{as } r\downarrow 0.
\end{align}
Y.~Peres and B.~Solomyak \cite[Theorem 2.2]{PeresSolomyak02} proved that for a specific class of $1$-Ahlfors-regular \emph{random} grid fractals $S$ the above holds in expectation, that is, 
\[
    \E[\Fav(S(r))]\sim \frac{1}{\log(1/r)}\,.
\]
The same estimate was obtained by S.~Zhang \cite{Zhang20} and by D.~Vardakis and A.~Volberg \cite{VardakisVolberg24}
for two other random models of $1$-Ahlfors-regular Cantor sets, suggesting some kind of universality. However, none of these works address the a.s.\ behavior of $\Fav(S(r))$ as $r\downarrow 0$. In fact, this is left as an open problem in \cite[Question 7.4]{PeresSolomyak02} for the model considered in that paper and, more generally, it remained unknown until this day whether there are $1$-sets satisfying \eqref{eq:minimal-fav}. 

In this paper, we solve this problem by confirming that \eqref{eq:minimal-fav} holds for a wide class of stochastically self-similar Cantor sets, including those considered in \cite{PeresSolomyak02,VardakisVolberg24}. Moreover, we will derive a formula for $\lim_{r\downarrow 0}\log(1/r)\Fav(S(r))$ showing, in particular, that the limit exists almost surely.

\begin{figure}[h]
\centering
\includegraphics[page=1,width=0.12\textwidth]{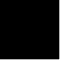}
\qquad
\includegraphics[page=2,width=0.12\textwidth]{figures_four_corner.pdf}
\qquad
\includegraphics[page=3,width=0.12\textwidth]{figures_four_corner.pdf}
\caption{$K_0, K_1, K_2$ for the four-corner Cantor set.}
\label{fig:four corner Cantor}
\end{figure}

Peres and Solomyak \cite{PeresSolomyak02} also initiated the study of the decay of Favard length for (deterministic) self-similar sets. A particularly striking example is the four-corner Cantor set $K$ (see the beginning of \cite{PeresSolomyak02} for the definition, or see \Cref{fig:four corner Cantor}). Let $K_n$ denote the level $n$ approximation of $K$ (which roughly corresponds to $K(4^{-n})$. Peres and Solomyak gave explicit, very slowly decaying upper bounds for $\Fav(K_n)$. In a significant breakthrough, F.~Nazarov, Y.~Peres and A.~Volberg \cite{NPS2010} proved the power upper bound $\Fav(K_n)\lesssim n^{-1/C}$ for some explicit $C$. (See, e.g., \cite{BondLabaVolberg14, BondVolberg10, BondVolberg12, LabaMarshall22, Tao09} for some additional results on upper bounds on the Favard length of self-similar sets. The recent work of C.~Marshall \cite{Marshall25} provides further extensions, including higher-dimensional generalizations, and additional references. See also the survey articles \cite{Laba15, Taylor24}.) 

On the other hand, M.~Bateman and A.~Volberg \cite[Theorem 1]{BatemanVolberg10} proved that \eqref{eq:minimal-fav} fails for $K$; more precisely, they showed that
\[
\Fav(K_n) \gtrsim \frac{\log n}{n}.
\]
This may appear surprising in light of the results described previously on random Cantor sets. However, we are able to show that the key feature of $K$ that produces this behavior is not its determinism, but a certain column alignment: we introduce a class of random Cantor sets which have a similar column alignment and, in \Cref{sec:bateman-volberg}, show that the Bateman--Volberg argument extends to them with small modifications.

\subsection{Main results}\label{sec:stat}

For simplicity, we only consider \emph{1-Ahlfors-regular grid models} in this introductory section. (See Sections \ref{sec:expected_to_almost_sure}, \ref{sec:expected_ub}, and \ref{sec:sharp_asymptotic} for more general statements.) Fix $L \geq 2$, and let $\cD_n$ be the set of squares obtained by dividing the unit square into an $L^n \times L^n$ grid.

By a 1-Ahlfors-regular grid model, we mean the following.

\begin{enumerate}
\item Subdivide $S_0 := [0,1]^2$ into a $L \times L$ grid of squares. Choose $L$ of them at random according to a probability distribution $\X$ on $\binom{\cD_1}{L}$. The union of the chosen squares yields a subset $S_1 \subset S_0$.
\item Repeat the selection step above \emph{independently} in each of the $L$ squares in $S_1$, using the same distribution $\X$, to get a set $S_2 \subset S_1$, which is made of up $L^2$ squares from $\cD_2$.
\item Repeat this indefinitely to get a nested sequence of random sets $S_0 \supset S_1 \supset S_2 \supset \cdots$ and a limiting set $S = \bigcap_n S_n$. 
\end{enumerate}

We introduce the following terminology to describe a random grid model $\X$:

\begin{itemize}
\item We say $\X$ is \emph{uniform} 
if each $Q \in \cD_1$ satisfies $\P[Q \subset S_1] = \frac{1}{L}$.
\item We say $\X$ is \emph{vertically degenerate} if $\X$ selects one square in each column almost surely. (See \Cref{fig:deg} for an example.)
\item We say $\X$ is \emph{horizontally degenerate} if $\X$ selects one square in each row almost surely.
\item  We say $\X$ is \emph{non-degenerate} if it is neither vertically nor horizontally degenerate. (See \Cref{fig:Ahlfors__grid_non_deg} for an example.)
\end{itemize}

If $\X$ is vertically degenerate, then  $\proj_0(S_n) = [0,1]$ almost surely for all $n$, i.e., there is no decay when projecting onto the $x$-axis. In this situation, we prove that we have an expected decay of $\frac{\log n}{n}$ instead of $\frac{1}{n}$. Furthermore, we show that vertical degeneracy and horizontal degeneracy are the \emph{only} causes of $\frac{\log n}{n}$ decay. In \Cref{sec:expected_ub} we prove the following.

\begin{theorem}
\label{theorem:random-favard}
Let $\X$ be an Ahlfors-regular uniform grid model. 
\begin{enumerate}
    \item If $\X$ is non-degenerate, then $\E \Fav(S_n) \sim \frac{1}{n}$.
    \item If $\X$ is degenerate, then $\E \Fav(S_n) \sim \frac{\log n}{n}$.
\end{enumerate} 
(The implied constants can depend on the model $\X$.)
\end{theorem}

Since the random model considered in \cite{PeresSolomyak02} is a non-degenerate Ahlfors-regular uniform grid model, \Cref{theorem:random-favard} is an extension of their result.

While \Cref{theorem:random-favard} concerns the expected value of the Favard length, it does not immediately imply any almost-sure behavior for a particular realization of the random set $S = \bigcap_n S_n$. By using the stochastic self-similarity of $S$ and concentration inequalities, we prove the following in \Cref{sec:expected_to_almost_sure}.

\begin{theorem}
\label{theorem:random-favard-pointwise}
If $\X$ is an Ahlfors-regular uniform grid model, then $\displaystyle\lim_{n\to\infty}  \tfrac{\Fav(S_n)}{\E \Fav(S_n)} = 1$ almost surely.
\end{theorem}

Combining Theorems \ref{theorem:random-favard} and \ref{theorem:random-favard-pointwise} immediately implies \eqref{eq:minimal-fav} for the said models.

The next theorem concerns the limiting values. We prove the following in \Cref{sec:sharp_asymptotic}. We note that \Cref{theorem:random-favard-pointwise} holds very generally and without the uniform grid assumption (See \Cref{sec:expected_to_almost_sure}), while in Theorems \ref{theorem:random-favard} and \Cref{theorem:lim-n-EFav} this is a key property.

\begin{theorem}
\label{theorem:lim-n-EFav}
Let $\X$ be a non-degenerate Ahlfors-regular uniform grid model. Then
\[\lim_{n \to \infty} n\E\Fav(S_n)=\int_0^\pi \frac{2}{V(\theta)}\,d\theta\,,\]
where
\[
V(\theta)
=
\int_\R \Var \left(L |\proj_\theta^{-1}(t) \cap S_1|\right) \, dt
\,.
\]
Combining this with \Cref{theorem:random-favard-pointwise} gives
\[\lim_{n \to \infty} n\Fav(S_n)=\int_0^\pi \frac{2}{V(\theta)}\,d\theta \qquad\text{almost surely}.\]
Furthermore,
\[\lim_{r \downarrow 0}\log_L(1/r)\Fav(S(r))=\int_0^\pi \frac{2}{V(\theta)}\,d\theta \qquad\text{almost surely}.\]
\end{theorem}

\subsection{Ideas of proof and connections to branching processes}
\label{section:intro ideas of proof}

Our proof of \Cref{theorem:random-favard} is an extension of the proof of \cite{PeresSolomyak02}, which in turn is inspired by the theory of branching processes, in particular, \emph{critical Galton--Watson processes}.  A Galton--Watson process is a probabilistic model of population growth: it begins with a single ancestor (the root), and each individual independently produces a random number of offspring according to a fixed probability distribution on $\mathbb{Z}_{\geq 0}$. The process is \emph{critical} if the expected number of offspring for each individual is $1$ and if the offspring distribution itself is not almost surely equal to 1. The following are key properties of critical Galton--Watson processes:
\begin{enumerate}
    \item $\P[\text{the tree survives to generation $n$}] \sim \frac{1}{n}$. (By ``survives to generation $n$'', we mean that the tree has at least one node in the $n$th generation.)
    \item Kolmogorov's estimate (e.g., \cite[Theorem 12.7]{LyonsPeres16}) states that 
    \begin{align*}
        \lim_{n \to \infty} n\, \P[\text{the tree survives to generation $n$}] = \sigma^2 / 2
        \,,
    \end{align*} 
    where $\sigma^2$ is the variance of the offspring distribution.
\end{enumerate}
Of course, the second fact implies the first. These two facts can be seen as analogs of \Cref{theorem:random-favard} and \Cref{theorem:lim-n-EFav}. Let us analyze this more closely. Observe that 
$
\proj_\theta(S_n) = \{t \in \R: \proj_\theta^{-1}(t) \cap S_n \neq \emp\}.
$ 
As we vary $\theta$ and $t$, the set $\ell(\theta,t) := \proj_\theta^{-1}(t)$ varies over all lines in $\R^2$. Thus, $\Fav(S_n)$ measures ``how many'' lines intersect $S_n$. By Fubini's theorem,
\begin{align}
\label{eq:expected-Fav-prob}
\E \Fav(S_n)
&=
\int_{0}^\pi
\int_\R
\P[
\ell(\theta,t) \cap S_n \neq \emp
]
\, dt 
\, d\theta.
\end{align}
For a fixed line $\ell$, the set of $L$-adic squares intersecting it naturally forms a tree structure:
\[
\{Q \in \cD_n : Q \cap \ell \neq \varnothing\}\,.
\]
The random process of generating $S_0 \supset S_1 \supset S_2 \supset \cdots$ can be viewed as a stochastic pruning of this tree:
\[
\{Q \in \cD_n : Q \cap \ell \neq \varnothing, Q \subset S_n\}\,.
\]
The integrand in \eqref{eq:expected-Fav-prob} is then the probability that this pruned tree survives to generation $n$. 

A key challenge is that unlike in Galton--Watson processes, the branching behavior in our situation depends on the geometry of line-square intersections (for each line, it is a different multi-type Galton-Watson process with uncountably many types). Despite this, we show that for most lines $\ell$, the corresponding tree behaves closely enough to a critical Galton--Watson process to obtain the upper bound $\P[\ell \cap S_n \neq \emp] \lesssim_{\theta} \frac{1}{n}$. In the non-degenerate case, the implied constant is uniform in $\theta$, which leads to the first part of \Cref{theorem:random-favard}. In the vertically degenerate case, the implied constant blows up as $\theta \to 0$; a careful analysis of this blow-up yields the second part of \Cref{theorem:random-favard}.

To analyze the limiting behavior of $\P[\ell\cap S_n\neq\varnothing]$ more closely, we use the identity 
\[
    |\ell\cap[0,1]^2|=L^n\E\bigl[|\ell\cap S_n| \bigm| \ell\cap S_n \neq \emp \bigr] \cdot\P[\ell\cap S_n \neq \emp]\,,
\]
thus reducing the problem to studying the asymptotic behavior of $nL^n\E[|\ell\cap S_n|\mid \ell\cap S_n\neq\varnothing]$. While for a fixed line this limit may fail to exist, we are able to extract a nested sequence of squares $[0,1]^2=Q^{(0)}\supset Q^{(1)}\supset Q^{(2)}\supset\cdots$ such that, for most lines in a given direction, ``the average of the sceneries'' $\ell\cap S_i\cap Q^{(i)}\setminus Q^{(i+1)}$ normalized back to the unit square is, up to a negligible error, the sum of independent predefined random variables of finitely many types. This allows us to apply a concentration inequality, resulting in an analog of Kolmogorov's estimate for critical Galton--Watson processes.

To prove  \Cref{theorem:random-favard-pointwise}, we use a simple but seemingly yet-unnoticed idea: let $m=\lceil 4 \log_L n\rceil$. Condition on $S_m$. Then $S_n$ is a union of $L^{m-n}$ IID translated copies of random sets with the same law as $L^{-m}S$. Favard length is subadditive, and so (always conditioned on $S_m$) we have that $S_n$ is bounded below by the sum of $L^{n-m}$ IID random variables with the same law as $L^{-m}\Fav(S)$. We can then use standard concentration inequalities together with regularity of the sequence $\E(\Fav(S_n))$ (observed by Bongers \cite[Theorem 3.1]{Bongers19} in the case of the four corner Cantor set) to obtain $\limsup_{n\to\infty} \Fav(S_n)/\E(\Fav(S_n))\le 1$ and, with some additional work, also the matching lower bound.

\subsection{Structure of the paper and further results}

In Section \ref{sec:model}, we introduce the class of models we investigate in the paper and provide several examples. 

In Section \ref{sec:expected_to_almost_sure}, we state and prove Theorem \ref{thm:almost-sure-convergence}, which relates almost-sure decay of Favard length for a wide class of models to expected decay. Together with \Cref{thm:expected-decay non-deg}, it yields \Cref{theorem:random-favard-pointwise}.

Section \ref{sec:expected_ub} is devoted to the proofs of Theorems \ref{thm:expected-decay non-deg} and \ref{thm:expected-decay deg}, which together generalize \Cref{theorem:random-favard} to a wider class of grid models.

In Section \ref{sec:sharp_asymptotic}, we establish the sharp asymptotic value of $n\E(\Fav(S_n))$ (and therefore of $n\Fav(S_n)$ a.s.) for non-degenerate grid models: see \Cref{thm:limit}.

In Appendix \ref{sec:bateman-volberg}, we adapt the Bateman-Volberg lower bound for $\Fav(K_n)$ to degenerate grid models, and finally in Appendix \ref{sec:elementaryprob} we have collected some well-known concentration inequalities which we use in our proofs.

\subsection{Acknowledgments}

We thank Benjamin Foster and Jason Hou for pointing out some small errors in an earlier version and for other helpful feedback. We thank the referee for many helpful comments and suggestions that improved the exposition.

\section{Preliminaries}
\label{sec:model}

\subsection{General notation}
\label{subsec:notation}

We use the standard $O(\cdot)$ notation, using subscripts to indicate any dependencies of the constants. \textbf{We always allow the implicit constants to depend on the model under consideration}. The following are all synonymous: $A = O(B), B = \Omega(A), A \lesssim B$, and $B \gtrsim A$. We write $A\sim B$ to denote $A=O(B)=O(A)$.  

Given $\theta\in [0,\pi]$ and $t\in\R$, we denote  
\[
    \ell(\theta,t) = \proj_\theta^{-1}(t)\,.
\]
That is, $\ell(\theta,t)$ is the line with direction  $(\cos \theta, \sin \theta)$ and containing $t (-\sin\theta, \cos\theta)$.

We denote $[j]=\{1,\ldots,j\}$ and we use $|\cdot|$ to denote the length (Lebesgue measure) of a linear set. 

\subsection{Stochastically self-similar fractals}
\label{subsec:GW}

We recall the construction of a special case of a class of random fractals going back to Falconer \cite{Falconer86}; see also \cite[Section 15.3]{LyonsPeres16}. We refer to these as \emph{stochastically self-similar fractals}, but note that this term is often used in the literature for more general constructions.

Let $J_0\subset\R^2$ be a closed bounded domain (the closure of its interior) with finitely many connected components. Let $\X=(\rvL; Q_1,\ldots, Q_\rvL)$ be a random variable with the following features:

\begin{enumerate}[label=\textup{(\alph*)}]
    \item \label{it:aa} $\rvL\in\{0,1,2,\ldots\}$ is a bounded random variable. 
    \item \label{it:bb} The sets $Q_i=h_i(J_0)$ are homothetic copies of $J_0$ that are contained in $J_0$.
    \item \label{it:cc} There is $\rho\in (0,1)$ such that $\E[\rvL]=\rho^{-1}$  and the contraction ratio of $h_i$ is $\rho$ for all $1\le i\le \rvL$.
\end{enumerate}

\begin{remark}
 Pick $M$ so that $\rvL\le M$ a.s. We may then identify the range of $\X$ with $\{0,\ldots,M\}\times\mathcal{M}^{M}$, where $\mathcal{M}$ is the three dimensional real manifold of homotheties on $\R^2$. To avoid measurability issues, we assume $\X$ is Borel-measurable. Most of the time, we will work with concrete naturally defined models (see \Cref{sec:stat}), where the marginals of $\X$ are either discrete or absolutely continuous with respect to the Lebesgue measure.   
\end{remark}
 
Given $J_0$ and $\X$, we define a nested sequence of random compact sets $S_0 \supset S_1 \supset S_2 \supset \cdots$ as follows. Let $\cS_0=\{J_0\}$ and $S_0 = J_0$. Once $\cS_n$ and $S_n=\cup\cS_n$ are defined, where $\cS_n$ is a family of homothetic copies of $J_0$, we let
\[
    \cS_{n+1} = \bigcup_{R\in\cS_n} \{ h_i^R(R) : i=1, \ldots, \rvL_R \}
    \qquad\text{and}\qquad S_{n+1} = \cup \cS_{n+1} \, ,
\]
where $(\rvL_R; h_1^R(J_0),\ldots, h_{\rvL_R}^R(J_0))$ are independent realizations of $\X$, also independent of $\cS_n$. Let 
\begin{align}
    \label{eq:def S}
    S = \bigcap_{n=0}^\infty S_n
\end{align}
be the random limit set. If $\P[\rvL=0]>0$, then $\P[S=\varnothing]>0$. Falconer \cite{Falconer86} showed that, conditioned on $S\neq\varnothing$, the Hausdorff and box-counting dimension of $S$ both equal $1$ almost surely (this is due to the assumption $\E[\rvL]=\rho^{-1}$ in \ref{it:cc}).

Recall that $\cS_n$ is the collection of homothetic copies of $J_0$ making up $S_n$. Let 
\[
    Z_n = \rho^n \# \mathcal{S}_n. 
\]
Then $Z_n$ is a martingale. By the martingale convergence theorem, it has an a.s.\ finite limit $Z$. Furthermore, a.s.\ $Z=0$ if and only if $S=\varnothing$. We refer to \cite[Section 5.1]{LyonsPeres16} for this and other basic facts about Galton-Watson processes.

We now introduce some sub-classes of this construction, see Figures \ref{fig:Ahlfors__grid_non_deg}---\ref{fig:deg}:
\begin{itemize}
    \item If $\rvL$ is a.s.\ constant and the interiors of the sets $Q_i$ in \Cref{it:bb} are pairwise disjoint, then we say that the model is \emph{Ahlfors-regular}, since it is indeed simple to check that in this case the set $S$ is $1$-Ahlfors regular. Note that, in this case, $Z\equiv 1$ and $\rvL=\rho^{-1}$.
    \item If $J_0=[0,1]^2$ and there is $L \geq 2$ such that the sets $Q_i$ are distinct elements of the grid $\{ [\frac{i-1}{L},\frac{i}{L}]\times [\frac{j-1}{L},\frac{j}{L}]: i,j \in [L]\}$, we say that the model is a \emph{grid model}. 
\end{itemize}

\subsection{Grid models}
\label{subsec:notation-gridmodels}

Fix $L \geq 2$. Let 
\begin{align}
\label{eq:def D_n}
\cD_n = \left\{ [\tfrac{i-1}{L^n},\tfrac{i}{L^n})\times [\tfrac{j-1}{L^n},\tfrac{j}{L^n}): i,j \in [L^n]\right\};
\end{align}
i.e., $\cD_n$ is the set of half-open squares obtained by dividing the unit square into an $L^n \times L^n$ grid. (See \Cref{remark:half-open} below about the use of half-open squares.) It will be helpful to give an ``address'' to each of these squares. Let $\Lambda= [L]^2$. Given
\[
    \bm\lambda = (\lambda_1, \ldots, \lambda_n) = \bigl( (i_1,j_1),\ldots,  (i_n,j_n)\bigr)\in\Lambda^n\,,
\]
we let $Q(\bm\lambda)$ denote the square in $\cD_n$ with top-right corner at $\sum_{m=1}^n  (i_m,j_m) L^{-m}$. This gives a bijection $\Lambda^n \to \cD_n$. We will therefore sometimes abuse notation slightly and identify $\Lambda^n$ with $\cD_n$. Given $\bm\lambda \in \Lambda^n$ and $0 \leq i \leq n$, we let $\bm\lambda^{(i)} = (\lambda_1, \ldots, \lambda_i)$, so that $Q(\bm\lambda) \subset Q(\bm\lambda^{(i)})$.

\begin{definition}\label{def:sq-uniform}
    Let $\X$ be a grid model. We say $\X$ is \emph{uniform} if 
    \begin{align}
    \label{eq:squarewise-uniform}
        \P[Q \in \cS_1] = \frac{1}{L} \qquad\text{for all $Q \in \cD_1$} \,.
    \end{align}
\end{definition}

\begin{definition}[Degeneracy]\label{def:degenerate}
    Let $\X$ be a grid model.
    \begin{itemize}
        \item We say $\X$ is \emph{vertically degenerate} if $\X$ selects one square in each column almost surely. (See \Cref{fig:deg} for an example.) Equivalently, for every instance of $\cS_1$,
        \[
            \P\bigl[\text{there exist } i,j_1,j_2 \in [L] \text{ such that } j_1\neq j_2 \text{ and } Q(i,j_1), Q(i,j_2)\in \cS_1\bigr] = 0 \,.
        \]
        \item We say $\X$ is \emph{horizontally degenerate} if $\X$ selects one square in each row almost surely.
        \item  We say $\X$ is \emph{non-degenerate} if it is neither vertically nor horizontally degenerate. (See \Cref{fig:Ahlfors__grid_non_deg} and \Cref{fig:perco} for examples.)
    \end{itemize}
\end{definition}

\begin{remark}%
\label{remark:half-open}
    In the general setting (\Cref{subsec:GW}), the set $J_0$ is by definition closed. However, in Sections \ref{sec:expected_ub} and \ref{sec:sharp_asymptotic}, it will be more convenient to use half-open squares (see \Cref{remark:reminder half open}). This is why we defined $\cD_n$ using half-open squares in \eqref{eq:def D_n}. Accordingly, in this situation, $\cS_n$ is a set of half-open squares, and $S_n = \cup \cS_n$ is a union of half-open squares. Note that the quantity $|\proj_\theta (S_n)|$ does not depend on whether we use closed squares or half-open squares. The only modification required when using half-open squares is to replace \eqref{eq:def S} by $S = \bigcap_{n=0}^\infty \overline{S_n}$.
\end{remark}

\subsection{Examples and figures}

\begin{example}
The Peres--Solomyak model \cite{PeresSolomyak02} is a non-degenerate Ahlfors-regular grid model with $L = 4$. We view the $4 \times 4$ grid as consisting of four $2 \times 2$ sub-grids. In each $2 \times 2$ sub-grid, we choose one square uniformly at random and independently of the others.
\end{example}

\begin{example}
\label{example:perco}
Fractal percolation is a non-degenerate uniform grid model that is not Ahlfors-regular. Fix $L \geq 2$. Each square in the $L \times L$ grid is independently chosen with probability $1/L$. Thus, $\rvL$ is $\mathsf{Binomial}(L^2,L^{-1})$. See \Cref{fig:perco}.
\end{example}

\begin{example}
\label{example:VV-discs}
The Vardakis--Volberg disc model \cite{VardakisVolberg24} is an Ahlfors-regular model. $J_0$ is the unit disc. Fix $L \geq 3$. We let $\rvL = L$ and $\rho = 1/L$. We let $Q_1$ be a disc of radius $1/L$ internally tangent to $J_0$ and placed uniformly at random. Then we let $Q_2, \ldots, Q_L$ be such that $Q_1, \ldots, Q_L$ are equally spaced and all internally tangent to $J_0$. See \Cref{fig:VV-discs}.
\end{example}

\begin{figure}[h!]
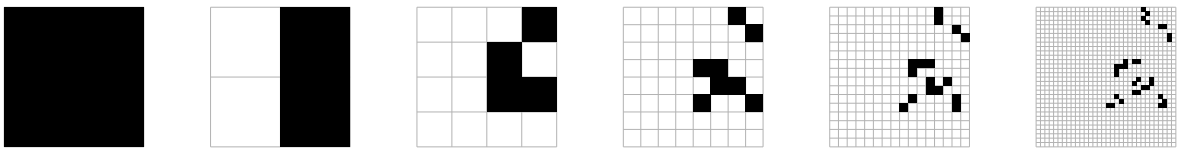

\centering
\includegraphics[page=1,width=0.12\textwidth]{figures_L2_nondeg.pdf}
\hfill
\includegraphics[page=2,width=0.12\textwidth]{figures_L2_nondeg.pdf}
\hfill
\includegraphics[page=3,width=0.12\textwidth]{figures_L2_nondeg.pdf}
\hfill
\includegraphics[page=4,width=0.12\textwidth]{figures_L2_nondeg.pdf}
\hfill
\includegraphics[page=5,width=0.12\textwidth]{figures_L2_nondeg.pdf}
\hfill
\includegraphics[page=6,width=0.12\textwidth]{figures_L2_nondeg.pdf}
\caption{$S_0$ through $S_5$ for a grid model: $L=2$, Ahlfors-regular, $\X$ is defined as a uniform choice among all pairs of disjoint $Q_1$ and $Q_2$.}
\label{fig:Ahlfors__grid_non_deg}
\end{figure}

\begin{figure}[h]
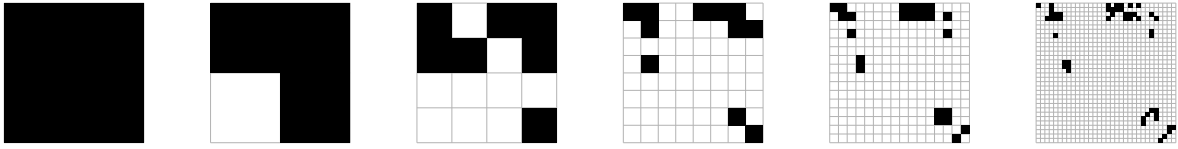

\centering
\includegraphics[page=1,width=0.12\textwidth]{figures_L2_non_Ahlfors.pdf}
\hfill
\includegraphics[page=2,width=0.12\textwidth]{figures_L2_non_Ahlfors.pdf}
\hfill
\includegraphics[page=3,width=0.12\textwidth]{figures_L2_non_Ahlfors.pdf}
\hfill
\includegraphics[page=4,width=0.12\textwidth]{figures_L2_non_Ahlfors.pdf}
\hfill
\includegraphics[page=5,width=0.12\textwidth]{figures_L2_non_Ahlfors.pdf}
\hfill
\includegraphics[page=6,width=0.12\textwidth]{figures_L2_non_Ahlfors.pdf}
\caption{$S_0$ through $S_5$ for a grid model (fractal percolation) with $L=2$. Each square is chosen independently with probability $1/2$. See \Cref{example:perco}.}
\label{fig:perco}
\end{figure}

\begin{figure}[h]
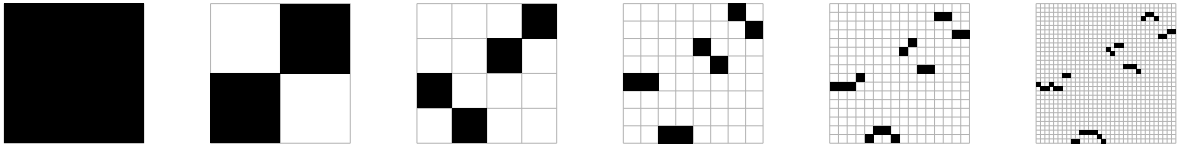

\centering
\includegraphics[page=1,width=0.12\textwidth]{figures_L2_deg.pdf}
\hfill
\includegraphics[page=2,width=0.12\textwidth]{figures_L2_deg.pdf}
\hfill
\includegraphics[page=3,width=0.12\textwidth]{figures_L2_deg.pdf}
\hfill
\includegraphics[page=4,width=0.12\textwidth]{figures_L2_deg.pdf}
\hfill
\includegraphics[page=5,width=0.12\textwidth]{figures_L2_deg.pdf}
\hfill
\includegraphics[page=6,width=0.12\textwidth]{figures_L2_deg.pdf}
\caption{$S_0$ through $S_5$ for a grid model with $L=2$. $\X$ selects one square uniformly and independently in each column, resulting in a vertically degenerate model. See \Cref{def:degenerate}.}
\label{fig:deg}
\end{figure}

\begin{figure}[h]
\centering
\includegraphics[page=1,width=0.15\textwidth]{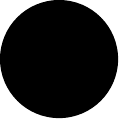}
\hfill
\includegraphics[page=2,width=0.15\textwidth]{figures_random_discs.pdf}
\hfill
\includegraphics[page=3,width=0.15\textwidth]{figures_random_discs.pdf}
\hfill
\includegraphics[page=4,width=0.15\textwidth]{figures_random_discs.pdf}
\caption{$S_0$ through $S_3$ for the model considered by Vardakis and Volberg, with $L=3$. See \Cref{example:VV-discs}.}
\label{fig:VV-discs}
\end{figure}

\begin{figure}[h]
\centering
\includegraphics[page=5,width=0.15\textwidth]{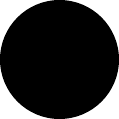}
\hfill
\includegraphics[page=6,width=0.15\textwidth]{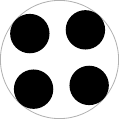}
\hfill
\includegraphics[page=7,width=0.15\textwidth]{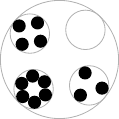}
\hfill
\includegraphics[page=8,width=0.15\textwidth]{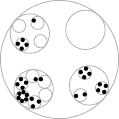}
\caption{$S_0$ through $S_3$ for a model that is the same as \Cref{example:VV-discs}, except that the number of discs $\rvL$ is uniform on $\{0,\ldots, 6\}$.}
\label{fig:discs-varying-number}
\end{figure}

\section{From expected to almost-sure Favard length}\label{sec:expected_to_almost_sure}

\subsection{The asymptotic limit}
 
 Throughout this section, we work with a fixed stochastically self-similar fractal model. The goal is to show that, under mild conditions, estimates on the expected decay of the Favard length of $S_n$ can be turned into almost-sure estimates. This is our main theorem in this direction:
\begin{theorem}[almost-sure convergence]
    \label{thm:almost-sure-convergence}
    Almost surely,
    \begin{align} \label{eq:as-limit-upper-bound}
        \limsup_{n\to\infty} \frac{\Fav(S_n)}{\E[\Fav(S_{n-m(n)})]} \le Z, \qquad\text{where } m(n) = \lfloor 4 \log_{\rho^{-1}} n\rfloor.
    \end{align}
    Furthermore, if 
    \[
        \sup_n \frac{\E[\Fav(S_{n})]}{\E[\Fav(S_{2n})]} <\infty,
    \]
    then 
    \begin{align*}
        \lim_{n\to\infty} \frac{\Fav(S_n)}{\E[\Fav(S_n)]} = \lim_{r\downarrow 0} \frac{\Fav(S(r))}{\E[\Fav(S(r))]}  = Z \qquad\text{almost surely}.
    \end{align*}
\end{theorem}

\subsection{Construction steps versus neighbourhoods of $S$}

We begin by showing that we can pass between $\Fav(S_n)$ and $\Fav(S(r))$ under a non-restrictive survival assumption.

We say that a model is \emph{surviving} if $\rvL\ge 1$ almost surely. This ensures that $S\cap Q\neq\varnothing$ for all $Q\in\cS_n$. We note that a model with $\P[\rvL=0]>0$ may always be reduced to a surviving model, in the sense that the law of the surviving model is the law of the original model conditioned on survival; see \cite[\S 5.7]{LyonsPeres16}. According to the following lemma, $\Fav(S_n)$ and $\Fav(S(\rho^n))$ are roughly comparable for surviving models. 

\begin{lemma} \label{lem:step-vs-neighbourhood}
Consider a surviving model. If $r \ge \diam(J_0)\rho^n$, then
\[
     \Fav(S(r))\ge \Fav(S_n)\,.
\]
If $r\le \varepsilon\rho^n$, then
\[
     \Fav(S(r)) \le (1+O(\varepsilon)) \Fav(S_n)\,.
\]
\end{lemma}
\begin{proof}
The model being surviving implies that $S_n\subset S(\diam(J_0)\rho^{n})$. This yields the first inequality.

For the second inequality, let $\tau > 0$ be such that $J_0$ contains a ball of radius $\tau$, and note that $\proj_{\theta}(S_n)$ is a union of disjoint intervals of length $\ge \tau\rho^{n}$.  It follows that there are $\le \tau^{-1}\rho^{-n}|\proj_{\theta}(S_n)|$ such intervals. For any set $A\subset\R$ which is a union of $K$ intervals and any $r>0$, we have $|A(r)\setminus A|\le 2 K r$. Hence, for any $r > 0$,
\[
|(\proj_\theta S)(r) \setminus \proj_\theta S_n|
\leq
|(\proj_\theta S_n)(r) \setminus \proj_\theta S_n|
\lesssim r \tau^{-1} \rho^{-n}|\proj_\theta S_n|
\,.
\]
Rearranging and integrating in $\theta$,
\[
\Fav(S(r))
\leq
(1+O(r \rho^{-n}))\Fav(S_n) \qquad\text{for any $r > 0$ and $n \in \N$}.
\]
This yields the claim.
\end{proof}

\subsection{Convexity of Favard length}

Throughout this section, we use the notation
\begin{align*}
    \mu_n(\theta) &= \E\bigl[|\proj_\theta(S_n)|\bigr]\,,\\
    \mu_n &= \E[\Fav(S_n)] = \int_{0}^{\pi} \mu_n(\theta)\,d\theta.
\end{align*}

\begin{lemma}[Convexity] \label{lem:fav-conv}
    For each $\theta$, the sequence $\mu_n(\theta)$ is convex (that is, $\mu_n(\theta) - \mu_{n+1}(\theta) \leq \mu_{n-1}(\theta) - \mu_n(\theta)$ for all $n$). Consequently, the sequence $\mu_n$ is convex.
\end{lemma}

\begin{proof}
    We adapt Bongers's proof for the standard four-corner Cantor set \cite[Theorem 3.1]{Bongers19} to our situation. Let $(\rvL;h_1(J_0),\ldots,h_\rvL(J_0))$ be a realization of $\X$, so that $\mathcal{S}_1=\{h_1(J_0),\ldots,h_\rvL(J_0)\}$. For each $i$, let $(S_n^{(i)})_n$ be an independent copy of $S_0\supset S_1\supset\cdots$, also independent of $(\rvL;h_1(J_0),\ldots,h_\rvL(J_0))$.  Then
    \[
        (S_n)_{n=1}^\infty \stackrel{d}{=} \left(\bigcup_{i=1}^\rvL  h_{i}(S_{n-1}^{(i)})\right)_{n=1}^\infty\,.
    \]
    Thus,
    \begin{align*}
        \proj_\theta S_n \setminus \proj_\theta S_{n+1}
        &\stackrel{d}{=} 
        \bigcup_{i=1}^\rvL  \proj_\theta h_{i}(S_{n-1}^{(i)})
        \setminus
        \bigcup_{i=1}^\rvL \proj_\theta h_{i}(S_{n}^{(i)})
        \\
        &\subset
        \bigcup_{i=1}^\rvL \proj_\theta h_{i}(S_{n-1}^{(i)}) \setminus \proj_\theta h_{i}(S_{n}^{(i)})\,.
    \end{align*}
    Taking expectations,
    \begin{align*}
        \mu_n(\theta) - \mu_{n+1}(\theta)
        &= 
        \E|\proj_\theta S_n \setminus \proj_\theta S_{n+1}|
        \\
        &\leq
        \E\left|\bigcup_{i=1}^\rvL  \proj_\theta h_i(S_{n-1}^{(i)}) \setminus \proj_\theta h_i(S_{n}^{(i)})\right|
        \\
        &\leq 
        \E \left[\sum_{i=1}^\rvL \left|\proj_\theta h_i(S_{n-1}^{(i)}) \setminus \proj_\theta h_i(S_{n}^{(i)}) \right|   \right]  
        \\
        &= \E \left[\rho \rvL   \left| \proj_\theta (S_{n-1}^{(i)}) \setminus \proj_\theta(S_{n}^{(i)})  \right|\right]
        \\
        &\stackrel{\ref{it:cc}}{=}
        \mu_{n-1}(\theta) - \mu_n(\theta)\,,
    \end{align*}
    using the independence of $\rvL$ and $S_j^{(i)}$ in the last equality. This shows that $\mu_n(\theta)$ is a convex sequence. To deduce the convexity of $\mu_n$, just integrate in $\theta$.
\end{proof}

\begin{lemma}\label{lem:conv}
    Suppose that the model has the property that
    \[
        M:=\sup_n \frac{\mu_n}{\mu_{2n}}<\infty\,.
    \]
    Then
    \begin{align}\label{eq:convex a_n+k}
    \mu_{n+k} \geq \left(1-\frac{2 M k}{n}\right) \mu_n \qquad  \text{ for all $n,k$}.
    \end{align}
    Consequently,
    \begin{align}\label{eq:convex limit}
    \lim_{n \to \infty} \frac{\mu_{n+f(n)}}{\mu_n} = 1
    \end{align}
    for all sequences $f(n)$ such that $f(n)/n\to 0$.
\end{lemma}
\begin{proof}
    This is a consequence of Lemma \ref{lem:fav-conv} and general facts about convex sequences. Indeed, $\mu_n$ is a positive, non-increasing, convex sequence, such that $\mu_{2n} \ge M^{-1} \mu_n$ for all $n$. These are all the properties of $(\mu_n)$ that we need:
    Let $n \in \N$. Since $\mu_n\ge 0$, $\mu_{n/2} - \mu_n \leq \mu_{n/2}$ so, writing $\mu_{n/2} - \mu_n$ as a telescoping sum and using the pigeonhole principle, there exists $i \in [n/2, n]$ such that $\mu_i - \mu_{i+1} \leq \frac{2\mu_{n/2}}{n}$. By convexity, $\mu_n - \mu_{n+k} \leq \frac{2k\mu_{n/2}}{n}$, and we deduce
    \[
        \frac{\mu_n - \mu_{n+k}}{\mu_n} \leq \frac{2k}{n} \cdot\frac{\mu_{n/2}}{\mu_n} \leq  \frac{2Mk}{n} \,.
    \]
    If we apply the inequality above with $k=f(n)$, we get the lemma.
\end{proof}

\subsection{Proof of Theorem \ref{thm:almost-sure-convergence}}

Recall the following from \Cref{subsec:GW}: 
\begin{enumerate}
    \item The random fractal is defined in terms of countably many independent copies of $\mathfrak{X}$. (See \Cref{subsec:GW}.) In the following lemmas, when we condition on $\cS_m$, this means that we condition on all the copies of $\mathfrak{X}$ which determine $\cS_m$. 
    \item $Z_m=\rho^m \#\mathcal{S}_m$ is the martingale associated to the process $\mathcal{S}_m$. 
\end{enumerate}

\begin{lemma}[Conditional upper tail bound]
For all $s > 0$ and all $m < n$,
\begin{align}
\label{eq:conditional-upper-tail}
    \P\bigl[\Fav(S_n)>Z_m \mu_{n-m} + s \bigm| \cS_m\bigr]
    \leq
    \exp\left(-\frac{\Omega(1)s^2}{Z_m\rho^m}\right)\,.
\end{align}
\end{lemma}

\begin{proof}
Note that Favard length is sub-additive: if $\{ A_i\}$ are finitely many (or countably  many, but we only need the finite case) Borel sets in $\R^2$, then
\[
    \Fav\left(\bigcup_i A_i\right) \le \sum_i \Fav(A_i).
\]
It follows that $\Fav(S_n)$ conditioned on $\cS_{m}$ is bounded above by the sum $Y$ of $\#\cS_m$ independent copies of $\rho^m \Fav(S_{n-m})$ (also independent of $\cS_{m})$. Note that $\E[Y] = Z_m \mu_{n-m}$. Since $0\le \rho^m\Fav(S_{n-m})\lesssim \rho^m$, Hoeffding's inequality (\Cref{lemma:hoeffding}) yields (for all $s > 0$)
\begin{align*}
    \P\bigl[\Fav(S_n)>Z_m \mu_{n-m} + s \bigm| \cS_m \bigr]
    \leq
    \P\bigl[Y>\E[Y] + s \bigr]
    \leq
    \exp\left(-\frac{\Omega(s^2)}{\#\cS_m (\rho^m)^2}\right)
    \,,
\end{align*}    
as desired.
\end{proof}

\begin{lemma}[Unconditional upper tail bound] \label{lem:upper-tail-bound}
For all $t > 0$ and all $n$, writing $m=\lfloor 4\log_{\rho^{-1}}(n)\rfloor$, we have
\begin{align} 
    \label{eq:Fav upper tail}
    \P\left[ \frac{\Fav(S_n)}{\mu_{n-m}} > Z_m + Z_{m}^{1/2} t   \right]
    &\le
    \exp(-\Omega(t^2)  n^2)\,.
\end{align}
\end{lemma}

\begin{proof}
We apply the previous lemma with $s = Z_m^{1/2} \mu_{n-m} t$ to get
\begin{align*}
    \P\left[ \frac{\Fav(S_n)}{\mu_{n-m}} > Z_m + Z_{m}^{1/2} t ~\middle|~ \cS_m \right]
    \stackrel{\eqref{eq:conditional-upper-tail}}{\leq}    
    \exp\left(-\frac{\Omega(t^2) \mu_{n-m}^2}{\rho^m}\right)
    \leq    
    \exp\left(-\Omega(t^2) n^2\right)
    \,,
\end{align*}
where in the second inequality we used $m = \lfloor 4\log_{\rho^{-1}}(n)\rfloor$ and $\mu_{n-m} \gtrsim \frac{1}{n-m} \sim \frac{1}{n}$. Since the bound is independent of $\mathcal{S}_m$, we obtain \eqref{eq:Fav upper tail}.
\end{proof}

\begin{lemma}[Unconditional lower tail bound] \label{lem:lower-tail-bound}
Suppose that the model has the property that
\[
    \sup_n \frac{\mu_n}{\mu_{2n}}<\infty\,.
\]
Then for all $t > 0$ and all $n$, writing $m=\lfloor 4\log_{\rho^{-1}}(n)\rfloor$, we have
\begin{align} 
    \label{eq:Fav lower tail}
    \P\left[ \frac{\Fav(S_n)}{\mu_{n-m}} < Z_m - t  \right]
    &\lesssim
    \frac{1}{t} \frac{\log n}{n} \,.
\end{align}
\end{lemma}

\begin{proof}
Fix $t > 0$ and $n$. Let $A$ denote the event $\frac{\Fav(S_n)}{\mu_{n-m}} < Z_m - t$, and let $p=\P\left[A \mid \cS_m\right]$. We decompose
\begin{align*}
    \E[\Fav(S_n) \mid \cS_m]
    =
    \E[\Fav(S_n)\bbone_A \mid \cS_m]
    +
    \E[\Fav(S_n)\bbone_{A^c} \mid \cS_m]
\end{align*}
For the two terms, we have
\begin{align*}
    \E[\Fav(S_n)\bbone_A \mid \cS_m] \leq (Z_m-t) \mu_{n-m} \cdot p
\end{align*}
and
\begin{align*}
    &\E[\Fav(S_n)\bbone_{A^c} \mid \cS_m]
    \\
    &=
    \int_0^\infty
    \P\bigl[\Fav(S_n)\bbone_{A^c} > s \bigm| \cS_m\bigr] \, ds
    \\
    &\le 
    (1-p)Z_m \mu_{n-m}
    +
    \int_{0}^\infty\P\bigl[\Fav(S_n)>Z_m \mu_{n-m} + s \bigm| \cS_m\bigr] \, ds 
    \\
    &\stackrel{\eqref{eq:conditional-upper-tail}}{\leq}
    (1-p)Z_m \mu_{n-m}
    +
    O((Z_m \rho^m)^{1/2})
    \,.
\end{align*}
Combining the three lines above and rearranging, we obtain
\begin{align*}
    t\,\P\left[ A ~\middle|~ \cS_m\right]
    \leq
    \frac{Z_m  \mu_{n-m} - \E[\Fav(S_n) \mid \cS_m]}{\mu_{n-m}} + O\left(\frac{(Z_m \rho^m)^{1/2}}{\mu_{n-m}}\right)  
\end{align*}
Taking expectations on both sides and using $\E[Z_m^{1/2}] \leq \E[Z_m]^{1/2} = 1$, $m=\lfloor 4\log_{\rho^{-1}}(n)\rfloor$, and $\mu_{n-m} \gtrsim \frac{1}{n}$, we have
\begin{align*}
    t\,
    \P\left[ A \right]
    \leq
    \frac{\mu_{n-m} - \mu_n}{\mu_{n-m}} + O\left(\frac{1}{n}\right)  
    \stackrel{\eqref{eq:convex a_n+k}}{\lesssim}
    \frac{\log n}{n}
\end{align*}
as desired. Note that we use the assumption $\sup_n \frac{\mu_n}{\mu_{2n}}<\infty$ in the final inequality above and note also that the reference is to the inequality \eqref{eq:convex a_n+k}, not Lemma \ref{lem:step-vs-neighbourhood}.
\end{proof}

\begin{proof}[Proof of \Cref{thm:almost-sure-convergence}]
 It follows from the upper tail bound (\Cref{lem:upper-tail-bound}) and Borel--Cantelli that, for any $t>0$, 
\[
    \limsup_{n\to\infty} \frac{\Fav(S_n)}{\mu_{n-m}} \le \limsup_{n\to\infty} (Z_m + Z_m^{1/2}t).
\]
Since $Z_m\to Z>0$ as $m\to\infty$ almost surely on non-extinction, we get the first claim \eqref{eq:as-limit-upper-bound} by letting $t\downarrow 0$.

Now suppose $\sup_n \frac{\mu_n}{\mu_{2n}} <\infty$. Then the first claim \eqref{eq:as-limit-upper-bound} and convexity (Lemma \ref{lem:conv}) imply that
\[
\limsup_{n\to\infty} \frac{\Fav(S_n)}{\mu_n}\le Z \qquad\text{almost surely}.
\]
Thus, it remains to show
\[
\liminf_{n\to\infty} \frac{\Fav(S_n)}{\mu_n}\ge Z \qquad\text{almost surely}.
\]
Let $m(n) = \lfloor 4\log_{\rho^{-1}}(n)\rfloor$. For $(k-1)^2 \leq n \leq k^2$, we have
\[
\frac{\Fav(S_n)}{\mu_n}
\geq
\frac{\Fav(S_{k^2})}{\mu_{k^2 - m(k^2)}}
\cdot
\frac{\mu_{k^2-m(k^2)}}{\mu_{(k-1)^2}}
\,.
\]
\Cref{lem:conv} implies $\frac{\mu_{k^2-m(k^2)}}{\mu_{(k-1)^2}} \to 1$, so
\[
\liminf_{n\to\infty} \frac{\Fav(S_n)}{\mu_n}
\geq
\liminf_{k \to \infty} \frac{\Fav(S_{k^2})}{\mu_{k^2-m(k^2)}}
\,.
\]
Note that $\sum_k \frac{\log (k^2)}{k^2} < \infty$. By the lower tail bound (\Cref{lem:lower-tail-bound}) and Borel--Cantelli, the right-hand side above is at least $Z$ almost surely. We have therefore shown that $\Fav(S_n)/\mu_n\to Z$ a.s. 

To pass from this to $\Fav(S_r)/\E[\Fav(S_r)]$, we use \Cref{lem:step-vs-neighbourhood} and the assumption $\mu_{n}=O(\mu_{2n})$. (Since $Z=0$ if and only if $S=\varnothing$ a.s., there is no loss of generality in passing to the surviving version of the model. This completes the proof.)
\end{proof}

\section{Expected decay of Favard length for grid models}\label{sec:expected_ub}

\subsection{The decay rate for uniform grid models}

The goal of this section is to prove the following results.
\begin{theorem} \label{thm:expected-decay non-deg}
    For a non-degenerate uniform grid model, 
    \[
        \E[\Fav(S_n)] \sim \frac{1}{n}\,.
    \]
\end{theorem}
\begin{theorem} \label{thm:expected-decay deg}
    For a degenerate uniform grid model, 
    \[
        \E[\Fav(S_n)] \sim \frac{\log n}{n}\,.
    \]
\end{theorem}

For both theorems, it is the upper bound that requires new ideas, while the lower bound follows from the bound \eqref{Fav_lb_deterministic} (if the model is Ahlfors regular) in Theorem \ref{thm:expected-decay non-deg} and using an adaptation of the the argument of Bateman--Volberg \cite{BatemanVolberg10} in Theorem \ref{thm:expected-decay deg}.

\begin{remark}%
\label{remark:reminder half open}
Recall that squares in this section are half-open, ensuring that $\cD_n$ consists of pairwise disjoint squares   (see \Cref{remark:half-open}). If we used closed squares instead, we would have to exclude an exceptional set of points or lines in several statements.
For example, equation \eqref{eq:P x in S_n} below would not hold for $x$ on the grid lines (i.e., $\bigcup_{Q \in \cD_n} \partial Q$). Consequently, in \eqref{eq:length-w-square-v1}, we would need to exclude the countably many vertical and horizontal lines that contain segments of the grid lines.
\end{remark}

\subsection{Configurations of subsquares}
\label{section:configuration of subsquares}

\subsubsection{Overview of definition}

We start the proofs of \Cref{thm:expected-decay non-deg} and \Cref{thm:expected-decay deg} with the central notion of ``large expected intersection'' introduced in \Cref{definition:set LEI} below. This definition captures the idea of the line-square intersections which ensure that the line has a large expected intersection with the random set $S_n$. The definition is somewhat technical, but its role will become clear in the course of the proof.

Before stating \Cref{definition:set LEI}, we sketch some of the relevant ideas and provide some motivation. We start by observing that the uniformity assumption \eqref{eq:squarewise-uniform} implies that
\begin{align}\label{eq:P x in S_n}
\P[x \in S_n] = \frac{1}{L^n}\qquad\text{for all }x \in [0,1)^2.
\end{align}
Thus, for all $n,\theta,t$,
\begin{align}\label{eq:length-w-square-v1}
\begin{split}
|\ell(\theta,t) \cap [0,1)^2|
&=
\E\bigl[ L^n |\ell(\theta,t) \cap S_n|\bigr]
\\
&=
\E\bigl[ L^n |\ell(\theta,t) \cap S_n| \bigm| \ell(\theta,t) \cap S_n \neq \varnothing\bigr] \cdot \P[\ell(\theta,t) \cap S_n \neq \varnothing]\,.
\end{split}
\end{align}
\Cref{thm:expected-decay non-deg} will be proved by showing that $\P[\ell \cap S_n]\lesssim \tfrac{1}{n}$ for `most' lines, which by \eqref{eq:length-w-square-v1} follows from the bound 
\begin{align}
\label{eq:motivation E condition ell cap S_n}
\E\bigl[L^n |\ell \cap S_n|\bigm| \ell \cap S_n\neq\varnothing\bigr]\gtrsim n.
\end{align}
For a simpler version of the argument, let us suppose that we would instead like 
\begin{align}
\label{eq:motivation E condition Q in S_n}
\E\bigl[L^n |\ell \cap S_n| \bigm| Q \in \cS_n \bigr] \gtrsim n\qquad\text{for all }Q \in \cD_n \text{ such that } \ell \cap Q \neq \varnothing.
\end{align}
Note that \eqref{eq:motivation E condition Q in S_n} does not actually imply \eqref{eq:motivation E condition ell cap S_n}. (The issue is that it is not true that $\E[X \mid A \lor B] \geq \min \{ \E[X \mid A], \E[X \mid B] \}$ for a random variable $X$ and events $A,B$.)

For a fixed $Q \in \cD_n$ with $\ell \cap Q \neq \varnothing$, write $Q = Q^{(n)} \subset Q^{(n-1)} \subset Q^{(n-2)} \subset \cdots \subset Q^{(1)} \subset Q^{(0)} = [0,1)^2$, where $Q^{(i)} \in \cD_i$ is the square at the $i$th level containing $Q$. Then we can decompose $[0,1)^2 \setminus Q$ into ``rings'' $Q^{(i)} \setminus Q^{(i+1)}$, leading to
\begin{align}
\label{eq:motivation decomposition into scales}
\E\bigl[L^n |\ell \cap S_n|\bigm| Q \in \cS_n\bigr] \geq \sum_{i=0}^{n-1}
\E\bigl[L^n |\ell \cap S_n \cap Q^{(i)} \setminus Q^{(i+1)})| \bigm| Q \in \cS_n \bigr].
\end{align}
We would like to estimate from below the right-hand side of \eqref{eq:motivation decomposition into scales}. \Cref{definition:set LEI} is the first step in identifying for which $Q$ the right-hand side is large. It only deals with a single term of the sum at a time. Later, in \Cref{lemma:expect multi scale}, we define an iterated form of this condition (see \eqref{eq:def cG large count}) which ensures that many terms in the sum in \eqref{eq:motivation decomposition into scales} are large.

\subsubsection{Definition and proofs} 

\begin{definition}[subsquares with large expected intersection]
    \label{definition:set LEI}
    
    Fix $\theta \in [0,\pi]$. For $\alpha \in \Lambda$, $k \in \N$, $c > 0$, we define $\LEI_k(\theta,\alpha,c) \subset \Lambda^k$ to be the set of $\bm\eta \in \{\alpha\} \times \Lambda^{k-1}$ such that
    \begin{align}
        \label{eq:def set V}
        \forall t \in \R,
        \qquad
        \ell(\theta,t) \cap Q(\bm\eta) \neq \varnothing
        \implies
        \E\bigl[ |\ell(\theta,t) \cap S_1 \cap Q(\alpha)^c| \bigm| Q(\alpha) \in \cS_1 \bigr] \geq \frac{c}{L}
        \,.
    \end{align}   
    We also define
    \[
        \LEI_k(\theta,c) = \bigcup_{\alpha \in \Lambda} \LEI_k(\theta,\alpha,c) 
        \subset \Lambda^k
        \,.
    \]
\end{definition}

Note that $\LEI_k(\theta,\alpha,c)$ depends on the model under consideration. 
The following lemma gives a sufficient condition for $\bm \eta \in \LEI_k(\theta,\alpha,c)$.

\begin{lemma}
\label{lemma:two squares V}
    Fix a direction $\theta$. Suppose $\alpha\neq \kappa\in\Lambda$ are such that $p := \P\bigl[Q(\kappa) \in \cS_1 \bigm| Q(\alpha) \in \cS_1 \bigr] > 0$, and suppose $\bm\eta \in \{\alpha\} \times \Lambda^{k-1}$ satisfies
    \begin{align}
        \label{eq:Q(eta) inside strip}
        Q(\bm\eta) \subset \bigcup_{t : |\ell(\theta,t) \cap Q(\kappa)| \geq c'/L} \ell(\theta,t) 
        \,.
    \end{align}
    (See \Cref{figure:Q(eta) inside strip}.) Then $\bm\eta \in \LEI_k(\theta,\alpha,c'p)$.
\end{lemma}

\begin{proof}
Suppose $\ell(\theta,t) \cap Q(\bm\eta) \neq \varnothing$. Since $|\ell(\theta,t) \cap S_1 \cap Q(\alpha)^c| \geq |\ell(\theta,t) \cap S_1 \cap Q(\kappa)| =  |\ell(\theta,t) \cap Q(\kappa)| \bbone_{Q(\kappa) \in \cS_1}$, we have
\begin{align*}
    &\E\bigl[ |\ell(\theta,t) \cap S_1 \cap Q(\alpha)^c| \bigm| Q(\alpha) \in \cS_1 \bigr]
    \\
    &\geq
    |\ell(\theta,t) \cap Q(\kappa)|  \, \P\bigl[ Q(\kappa) \in \cS_1 \bigm| Q(\alpha) \in \cS_1 \bigr]
    \stackrel{\eqref{eq:Q(eta) inside strip}}{\geq}
    \frac{c'p}{L}
\end{align*}
as desired.
\end{proof}

\begin{figure}[h]
\centering
\includegraphics[width=0.5\textwidth]{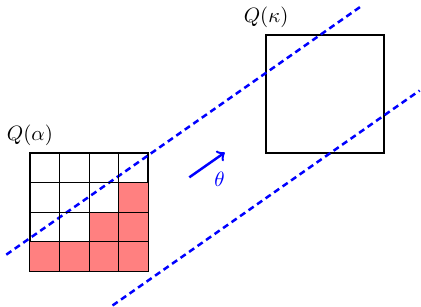}
\caption{The strip between the two dashed lines is the right-hand side of \eqref{eq:Q(eta) inside strip}. The shaded subsquares are the squares $Q(\bm\eta)$ which satisfy \eqref{eq:Q(eta) inside strip}.}
\label{figure:Q(eta) inside strip}
\end{figure}

The main utility of \Cref{definition:set LEI} is the following.

\begin{proposition}\label{prop:expected-line-survival kc}
    Suppose $\LEI_k(\theta,c) \neq \varnothing$ for some $k,\theta,c$. Then for all $n\in\N$, there is a set $E\subset\R$ of Lebesgue measure $\lesssim e^{-\Omega_{k,c}(n)}$, such that
    \begin{align}
        \label{eq:expected-line-survival-gen}
        \P\bigl[\ell(\theta,t)\cap S_n \neq\varnothing\bigr]\lesssim_{k,c} \frac{1}{n} \qquad\text{for all }t\in \R\setminus E\,.
    \end{align}
    Consequently,
    \begin{align}
        \label{eq:expected-projection-gen}
        \E|\proj_\theta S_n| \lesssim_{k,c} \frac{1}{n}
    \end{align}
\end{proposition}

The remainder of \Cref{section:configuration of subsquares} is the proof of \Cref{prop:expected-line-survival kc}. First, the following lemma states that \eqref{eq:def set V} automatically implies a similar property for $S_n$:

\begin{lemma}
    \label{lemma:expect single scale}
    Let $\bm\eta \in \LEI_k(\theta,\alpha,c)$. Then for all $n \geq 1$,
    \begin{align*}
        \forall t \in \R,
        \qquad
        \ell(\theta,t) \cap Q(\bm\eta) \neq \varnothing
        \implies
        \E\bigl[ |\ell(\theta,t) \cap S_n \cap Q(\alpha)^c| \bigm| Q(\alpha) \in \cS_1 \bigr] \geq \frac{c}{L^n}
        \,.
    \end{align*}    
\end{lemma}

\begin{proof}
To prove this lemma, it suffices to show that for all $n \geq 1$,
\begin{align}
\label{eq:prob x in S_n}
    \P(x \in S_n \mid Q(\alpha) \in \cS_1)
    =
    \frac{1}{L^{n-1}}
    \P(x \in S_1 \mid Q(\alpha) \in \cS_1)
    \qquad\text{for all $x \in Q(\alpha)^c$}.
\end{align}
Indeed, by integrating both sides of \eqref{eq:prob x in S_n} over $x \in \ell(\theta,t) \cap Q(\alpha)^c$, we obtain
\begin{align*}
    \E\bigl[ |\ell(\theta,t) \cap S_{n} \cap Q(\alpha)^c| \bigm| Q(\alpha) \in \cS_1 \bigr]
    =
    \frac{1}{L^{n-1}}\E\bigl[ |\ell(\theta,t) \cap S_{1} \cap Q(\alpha)^c| \bigm| Q(\alpha) \in \cS_1 \bigr]
    \,.
\end{align*}
Combining this with \eqref{eq:def set V}, we obtain the lemma.

To prove \eqref{eq:prob x in S_n}, fix $x \in [0,1)^2 \setminus Q(\alpha)$ and $n \geq 1$. Then there exists $\kappa \in \Lambda \setminus \{\alpha\}$ such that $x \in Q(\kappa)$. Since $x \in Q(\kappa)$, we have the following relations between events:
\begin{align*}
    x \in S_n \implies x\in S_1 \iff Q(\kappa) \in \cS_1.
\end{align*}
Hence,
\begin{align*}
    &\P(x \in S_n \mid Q(\alpha) \in \cS_1)
    \\
    &=
    \P(x \in S_n  \land Q(\kappa) \in \cS_1 \mid Q(\alpha) \in \cS_1)
    \\
    &=
    \P(x \in S_n \mid Q(\alpha) \in \cS_1 \land Q(\kappa) \in \cS_1)
    \P(Q(\kappa) \in \cS_1 \mid Q(\alpha) \in \cS_1)
    \\
    &=
    \P(x \in S_n \mid Q(\alpha) \in \cS_1 \land Q(\kappa) \in \cS_1)
    \P(x \in S_1 \mid Q(\alpha) \in \cS_1)
    \,.
\end{align*}
Note that conditioned on $Q(\kappa) \in \cS_1$, the random process inside $Q(\kappa)$ is independent of the event $Q(\alpha) \in \cS_1$. Thus, conditioned on $Q(\kappa) \in \cS_1$, the events $x \in S_n$ and $Q(\alpha) \in \cS_1$ are independent, which implies
\begin{align*}
    \P(x \in S_n \mid Q(\alpha) \in \cS_1 \land Q(\kappa) \in \cS_1)
    =
    \P(x \in S_n \mid Q(\kappa) \in \cS_1)
    \, .
\end{align*}
Furthermore, the random set $S_n \cap Q(\kappa)$ conditioned on $Q(\kappa) \in \cS_1$ has the same distribution as a homothetic copy of $S_{n-1}$, so
\begin{align*}
    \P(x \in S_n \mid Q(\kappa) \in \cS_1)
    =
    \P(h(x) \in S_{n-1})
    =
    \frac{1}{L^{n-1}}
    \, ,
\end{align*}
where $h$ is the homothety renormalizing $Q(\kappa)$ to $[0,1)^2$. This completes the proof of \eqref{eq:prob x in S_n}.
\end{proof}

The next lemma is a multiscale version of \Cref{lemma:expect single scale}. To state it, we first make the following definition, which counts the number of scales in which a particular configuration occurs.

\begin{definition}
\label{def:Count}
Given $\bm\lambda \in \Lambda^n$ and $V \subset \Lambda^k$ with $k \leq n$, we define $\Count(\bm\lambda, V)$ to be the number of substrings of $\bm\lambda$ of length $k$ that are in $V$, i.e.,
\begin{align}
\label{eq:def Count}
\Count(\bm\lambda, V)
=
\#\{i \in \{0, \ldots, n-k\} : (\lambda_{i+1}, \ldots, \lambda_{i+k}) \in V\}
\end{align}
Given $\bm\lambda \in \Lambda^n$ and $\bm\eta \in \Lambda^k$ with $k \leq n$, we also write $\Count(\bm\lambda, \bm\eta)$ for $\Count(\bm\lambda, \{\bm\eta\})$.
\end{definition}

Before stating the lemma, a word on terminology: In our applications, $\cG \subset \Lambda^n$ will be considered `a good set' and its complement $\cB$ will be the exceptional set of `bad squares'. In practise, this means that the parameter $M$ will be chosen so that most squares in $\Lambda_n$ will be good. 

\begin{lemma}
    \label{lemma:expect multi scale}
    Fix $k,\theta,\alpha,c$ and $n \geq 1$ and $M \geq 1$. Let
    \begin{align}
        \label{eq:def cG large count}
        \mathcal{G} = \left\{ \bm\lambda\in\Lambda^n: \Count(\bm\lambda, \LEI_k(\theta,\alpha,c)) \ge M \right\}\,.
    \end{align}
    Then for any $t \in \R$,
    \begin{align}
        \label{eq:expect cpM lemma}
        \E\bigl[ L^n |\ell(\theta,t) \cap S_n| \bigm|  \ell(\theta,t) \cap (\cup \cG) \cap S_n \neq \varnothing \bigr] \geq c M \, ,
    \end{align}
    and consequently,
    \begin{align}
        \label{eq:expect cpM lemma consequently}
        \P\left[\ell\cap (\cup\mathcal{G}) \cap S_n\neq\varnothing\right]\lesssim \frac{1}{cM} \,.
    \end{align}
    (Here, we identify $\Lambda^n$ with $\cD_n$, so that $\cup \cG$ is defined to be $\bigcup_{\bm \lambda \in \Lambda^n} Q(\bm\lambda)$.)
\end{lemma}

\begin{proof}
In this proof, we fix a total order $\prec$ on $\Lambda$ so that $\alpha$ is minimal. For each $n$, this yields a lexicographic order, also denoted by $\prec$, on $\Lambda^n$. In turn, the identification $\lambda\leftrightarrow Q(\lambda)$ yields an order on $\cD_n$, which we again denote by $\prec$.

Fix a line $\ell = \ell(\theta,t)$. For $\bm\lambda \in \cG$, let $\mathcal{E}(\bm\lambda)$ denote the event that $\bm\lambda = \min \{\bm\lambda' \in \cG : \ell \cap Q(\bm\lambda') \cap S_n \neq \varnothing\}$, where the minimum is defined in terms of $\prec$. Since the event $\ell \cap (\cup\cG) \cap S_n \neq \varnothing$ is the disjoint union of the events $\{ \mathcal{E}(\bm\lambda) : \bm\lambda \in \cG, \ell \cap Q(\bm\lambda) \neq \varnothing\}$, to prove \eqref{eq:expect cpM lemma}, it suffices to show
\begin{align}
    \label{eq:expect cpM step}
    \E\bigl[ L^n |\ell \cap S_n| \bigm|  \mathcal{E}(\bm\lambda) \bigr]
    \geq cM
    \qquad\text{for all $\bm\lambda \in \cG$ such that $\ell \cap Q(\bm\lambda) \neq \varnothing$}.
\end{align}

Fix $\bm\lambda \in \cG$ with $\ell \cap Q(\bm\lambda) \neq \varnothing$. Let 
\begin{align*}
    \cI = \{i\in \{0,\ldots,n-k\} : (\lambda_{i+1},\ldots,\lambda_{i+k}) \in \LEI_k(\theta,\alpha,c)\} 
    \,,
\end{align*} 
so that $\# \cI = \Count(\bm\lambda, \LEI_k(\theta,\alpha,c)) \geq M$. For all $0\le i\le n$, let $Q^{(i)}\in\cD_{i}$ be the square containing $Q(\bm\lambda)$. Since $S_0  = Q(\bm\lambda) \sqcup \bigsqcup_{i=0}^{n-1} (Q^{(i)} \setminus Q^{(i+1)})$,
\begin{align*}
    \E\bigl[ L^n |\ell \cap S_n| \bigm|  \mathcal{E}(\bm\lambda) \bigr]
    \geq
    \sum_{i=1}^n
    \E\bigl[ L^n |\ell \cap S_n \cap (Q^{(i)} \setminus Q^{(i+1)})| \bigm|  \mathcal{E}(\bm\lambda) \bigr]
    \,.
\end{align*}
Thus, to prove \eqref{eq:expect cpM step}, it suffices to show
\begin{align}
    \label{eq:expect cp I}
    \E\bigl[ L^n |\ell \cap S_n \cap (Q^{(i)} \setminus Q^{(i+1)})| \bigm|  \mathcal{E}(\bm\lambda) \bigr]
    \geq c
    \qquad\text{for all } i \in \cI.
\end{align}

Fix $i\in\mathcal{I}$ and let $h$ be the homothety renormalizing $Q^{(i)}$ to $[0,1)^2$. Since $\bm\eta := (\lambda_{i+1},\ldots,\lambda_{i+k}) \in  \LEI_k(\theta,\alpha,c)$ we have that $h(Q^{(i+1)})=Q(\alpha) $ and $h(Q(\bm\lambda)) \subset Q(\bm\eta)$.  Condition on $\mathcal{E}(\bm\lambda)$ and a realization of $\cS_{i}$. Clearly, $Q^{(i)}\in \cS_{i}$. Given this, by the independence in the model, the expectation we want to estimate in \eqref{eq:expect cp I} depends only on the process inside $Q^{(i)}$. Since $Q(\bm\lambda) \subset Q^{(i+1)}$, we know that $h(Q(\bm\lambda)) \subset Q(\alpha)$, and therefore $Q^{(i+1)}$ is the $\prec$-minimal square in $Q^{(i)}$ such that $Q^{(i+1)}\in \cS_{i+1}$.  As far as $S_{i+1}$ given the realization of $S_{i}$ is concerned, the fact that $\mathcal{E}(\bm\lambda)$ occurs gives no other information. Thus,
\begin{align*}
    &\E\bigl[ L^n |\ell \cap S_n \cap (Q^{(i)} \setminus Q^{(i+1)})| \bigm|  \mathcal{E}(\bm\lambda) \bigr]
    \\
    &=
    \E\bigl[ L^n |\ell \cap S_n \cap (Q^{(i)} \setminus Q^{(i+1)})| \bigm| Q^{(i+1)} \in \cS_{i+1} \bigr]
    \\
    &=
    \E\bigl[ L^{n-i} |h(\ell) \cap S_{n-i} \cap Q(\alpha)^c| \bigm| Q(\alpha) \in \cS_1 \bigr]
    \,,
\end{align*}
where in the second equality, we applied the homothety $h$ and used self-similarity of the model. Furthermore, since $h(Q(\bm\lambda)) \subset Q(\bm\eta)$ and $\ell \cap Q(\bm\lambda) \neq \varnothing$, it follows that $h(\ell) \cap Q(\bm\eta) \neq \varnothing$. Thus, we can apply \Cref{lemma:expect single scale} to obtain 
\begin{align*}
    \E\bigl[ L^{n-i} |h(\ell) \cap S_{n-i} \cap Q(\alpha)^c| \bigm| Q(\alpha) \in \cS_1 \bigr]
    \geq
    c,
\end{align*}
which proves \eqref{eq:expect cp I} and completes the proof of \eqref{eq:expect cpM lemma}.

Finally, 
\begin{align*}
1 \gtrsim |\ell \cap S_0|
&=
\E\bigl[ L^n |\ell \cap S_n|\bigr]
\\
&\geq
\E\bigl[ L^n |\ell \cap S_n| \bigm| \ell \cap (\cup \cG) \cap S_n \neq \varnothing\bigr] \cdot \P[\ell \cap (\cup \cG) \cap S_n \neq \varnothing]\,
\end{align*}
which proves \eqref{eq:expect cpM lemma consequently}.
\end{proof}

For the remaining squares $\cB = \cD_n \setminus \cG$, we will use the following lemma, which only takes into account the cardinality of $\cB$. Note that the conclusion of \Cref{lemma:markov} is nontrivial whenever $\#\cB = o(\#\cD_n)$. 

\begin{lemma}[Bound using Markov's inequality]
    \label{lemma:markov}
    Let $\cB \subset \cD_n$ and $q \in (0,1)$. There is a set $E\subset\R$ of Lebesgue measure $\lesssim \#\cB \cdot q^{-1} L^{-2n}$, such that
    \[ 
        \P\bigl[\ell(\theta,t)\cap (\cup \cB) \cap S_n \neq\varnothing\bigr] \leq q \,,\quad\text{for all }t\in \R\setminus E\,.
    \]
    Moreover, $E$ can be taken to be a union of intervals of length $\ge L^{-n}$.
\end{lemma}

\begin{proof}
By the union bound, 
\begin{align}
    \label{eq:union bound}
    \P\bigl[\ell(\theta,t)\cap (\cup \cB) \cap S_n \neq\varnothing\bigr]
    \leq
    \sum_{Q \in \cB}
    \P\bigl[\ell(\theta,t)\cap Q \cap S_n \neq\varnothing\bigr]
    =
    L^{-n} \sum_{Q \in \cB} \bbone_{\proj_\theta (Q)} (t)
    .
\end{align}
Define
\begin{align*}
    E' &= \{t \in \R : \sum_{Q \in \cB} \bbone_{\proj_\theta (Q)} (t) \geq q L^{n}\}
    \\
    E'' &= \{t \in \R : \sum_{Q \in \cB} \bbone_{3(\proj_\theta (Q))} (t) \geq q L^{n}\}
\end{align*}
where $3(\proj_\theta (Q))$ is the interval with the same center as $\proj_\theta (Q)$ and with three times the length. Since $|\proj_\theta(Q)| \geq L^{-n}$, we have $E := E'(L^{-n}) \subset E''$. We claim $E$ satisfies the desired properties. By definition, $E$ is a union of intervals of length $\geq L^{-n}$. Next, by Markov's inequality,
\[
    |E|
    \leq
    |E''|
    \leq
    q^{-1} L^{-n} \int_\R \sum_{Q \in \cB} \bbone_{3(\proj_\theta (Q))}(t) \, dt
    \lesssim
    \#\cB \cdot q^{-1} \cdot L^{-2n} 
    \,.
\]
Finally, if $t \in \R \setminus E \subset \R \setminus E'$, then by \eqref{eq:union bound}, we have
$
    \P\bigl[\ell(\theta,t)\cap (\cup \cB) \cap S_n \neq\varnothing\bigr] \leq q \,.
$
\end{proof}

\begin{proof}[Proof of \Cref{prop:expected-line-survival kc}]
    Since $\LEI_k(\theta,c) \neq \varnothing$, there exists $\alpha \in \Lambda$ such that $\LEI_k(\theta,\alpha,c) \neq \varnothing$. Fix $n$. Let $n_0 = \lfloor n/k \rfloor$, and define
    \begin{align*}
        \cG &= \left\{ \bm\lambda\in\Lambda^n: \Count(\bm\lambda, \LEI_k(\theta,\alpha,c)) \ge n_0 L^{-2k}/2 \right\}
        \\
        \cB &= \cD_n \setminus \cG.
    \end{align*}
    By \Cref{lemma:expect multi scale}, for all $t \in \R$,
    \[
        \P[\ell(\theta,t) \cap (\cup \cG) \cap S_n\neq\varnothing]\lesssim \frac{1}{c} \frac{2L^{2k}}{n_0} \sim_{k,c} \frac{1}{n}.
    \]
    
    Next, we will apply Markov's inequality (\Cref{lemma:markov}) with $\cB$; for that, we first need an upper bound on $\# \cB$. Fix an arbitrary $\bm\eta \in \LEI_k(\theta,\alpha,c)$. Observe that
    \[
        \cB 
        \subset
        \left\{ \bm\lambda\in\Lambda^n: \#\{i\in [n_0]:(\lambda_{(i-1)k+1},\ldots,\lambda_{ik}) = \bm\eta \} \le n_0 L^{-2k}/2 \right\}
    \]
    Consider $\lambda_1, \ldots, \lambda_n$ as IID uniform random variables taking values in $\Lambda$. Then the events $(\lambda_{(i-1)k+1},\ldots,\lambda_{ik})=\bm\eta$ for $i\in [n_0]$ are IID $\textup{Bernoulli}(L^{-2k})$ random variables (this is why we only consider multiples of $k$), so we can apply Chernoff's inequality (\Cref{lemma:chernoff}) to get
    \[
        \#\mathcal{B} \le 2 \exp\left(-\frac{n_0}{16 L^{2k}}\right)\cdot  \# \Lambda^n.
    \]
    By \Cref{lemma:markov}, there is a set $E \subset\R$ of Lebesgue measure $\lesssim e^{-\Omega_{k,c}(n)}$, such that
    \[ 
        \P\bigl[\ell(\theta,t)\cap (\cup \cB) \cap S_n \neq\varnothing\bigr] \lesssim \frac{1}{n} \quad\text{for all }t\in \R\setminus E.
    \]
    Since $\cD_n = \cG \cup \cB$, we have proved \eqref{eq:expected-line-survival-gen}.
    
    To deduce \eqref{eq:expected-projection-gen}, we simply observe
    \[
        \proj_{\theta} S_n = \{ t\in\R: \ell(\theta,t)\cap S_n\neq\varnothing\}
    \]
    so
    \begin{align*}
        \E |\proj_\theta S_n| 
        &= 
        \int_{\R} \P[\ell(\theta,t) \cap S_n\neq\varnothing] \, dt
        \leq
        \int_{\R \setminus E} \P[\ell(\theta,t) \cap S_n\neq\varnothing] \, dt
        +
        |E|
        \\
        &\lesssim_{k,c}
        \frac{1}{n}
        +
        e^{-\Omega_{k,c}(n)}
        \lesssim_{k,c}
        \frac{1}{n}\,.        
    \end{align*}
\end{proof}

\subsection{The non-degenerate case}
\label{sec:non-deg case}

Note that the implied constants in \Cref{prop:expected-line-survival kc} depend on the parameters $k$ and $c$ of the proposition. In the non-degenerate case, the following two lemmas allow us to find $k$ and $c$ depending only on the model such that the hypothesis of \Cref{prop:expected-line-survival kc} (i.e., $\LEI_k(\theta,c) \neq \varnothing$) holds for \emph{all} directions $\theta$. We state these lemmas so that they are also applicable in the degenerate case. 

We note that the role of the parameter $k$ is crucial here. In the degenerate case, it will be essential to consider all values of $k$. Also, in the nondegenerate case, it is possible that there exists $\theta$ such that  $\LEI_1(\theta,c)=\varnothing$ for all $c > 0$, but as it turns out, if $k\in\N$ and $c>0$ are chosen appropriately, then $\LEI_k(\theta,c)\neq\varnothing$ for all $\theta$.

\begin{lemma}[Existence of parameters for which $\LEI_k(\theta,c) \neq \varnothing$]
    \label{lemma:nonempty V local}
    Consider a uniform grid model.
    \begin{enumerate}[label=\textup{(\roman*)}]
        \item
        \label{item:nonempty V transversal}
        For every transversal direction $\theta \in (0,\pi/2) \cup (\pi/2,\pi)$,
        \begin{align}
            \label{eq:nonempty V nbhd}
            \text{there exist $k,c$ such that $\LEI_k(\theta,c) \neq \varnothing$.}
        \end{align}
        (The parameters $k,c$ can depend on the model.)
        \item 
        \label{item:nonempty V vertical}
        If the model is not vertically degenerate, then \eqref{eq:nonempty V nbhd} holds for the vertical direction $\theta = \pi/2$.
        \item 
        \label{item:nonempty V horizontal}
        If the model is not horizontally degenerate, then \eqref{eq:nonempty V nbhd} holds for the horizontal direction $\theta = 0$.
    \end{enumerate}
\end{lemma}

\begin{proof}
    To prove \ref{item:nonempty V transversal}, fix a transversal direction $\theta \in (0,\pi/2) \cup (\pi/2,\pi)$. By elementary geometric considerations, there exists $t_0$ such that $\ell_0:=\ell(\theta,t_0)$ does not hit any vertices of squares in $\mathcal{D}_1$, and 
    \[
        \# \{ Q \in \cD_1 : Q \cap \ell_0 \neq \varnothing \} \ge L+1.
    \]
    (This is where we use the assumption that $\theta$ is a transversal direction.) Since each square in $\cD_1$ is in $\mathcal{S}_1$ with probability $L^{-1}$, we deduce that
    \[
        \E\left[\# \{ Q \in \cS_1 : Q \cap \ell_0 \neq \varnothing \} \right]\ge 1+1/L\,.
    \]
    It follows that
    \[
        \P\bigl[\# \{ Q \in \cS_1 : Q \cap \ell_0 \neq \varnothing \}  \ge 2\bigr] > 0 \,.
    \]
    (In fact, it is easy to see the probability is at least $1/L^2$.) We can therefore pick two distinct $\alpha, \kappa \in \Lambda$ such that $\ell_0$ intersects both $Q(\alpha)$ and $Q(\kappa)$ and such that $\P[Q(\alpha) \in \cS_1 \land Q(\kappa) \in \cS_1] > 0$. By this fact, \Cref{lemma:two squares V}, and elementary geometry, there exist $k\in\N$, $c > 0$ such that $\LEI_k(\theta,\alpha,c) \neq \varnothing$. This proves \ref{item:nonempty V transversal}.

    To prove \ref{item:nonempty V vertical}, suppose the model is not vertically degenerate. Then there exist two distinct squares $Q(\alpha), Q(\kappa) \in \cD_1$ in the same column such that $\P[Q(\alpha) \in \cS_1 \land Q(\kappa) \in \cS_1] > 0$. Then by \Cref{lemma:two squares V}, we have
    \[
    \LEI_1(\pi/2,\alpha,c) = \{\alpha\} \qquad \text{where } c := \P[Q(\kappa) \in \cS_1 \mid Q(\alpha) \in \cS_1] > 0\,.
    \] 
    This proves \ref{item:nonempty V vertical}. The proof of \ref{item:nonempty V horizontal} is similar.
\end{proof}

\begin{lemma}[Upgrade to uniform parameters for which $\LEI_k(\theta,c) \neq \varnothing$]
\label{lem:non-deg-structure}
Let $\Theta \subset [0,\pi]$ be a compact set. Suppose that for all $\theta \in \Theta$, $\LEI_{k}(\theta,c) \neq \varnothing$ for some $k,c$ (possibly depending on $\theta$). Then there exist $k,c$ depending only on the model and on $\Theta$ such that for all $\theta \in \Theta$, $\LEI_k(\theta,c) \neq \varnothing$.
\end{lemma}

\begin{proof}[Proof of \Cref{lem:non-deg-structure}]
    Before proving the lemma, we first observe the following two properties of $\LEI_k(\theta,\alpha,c)$: 
    \begin{enumerate}
        \item Monotonicity in $k$ and $c$: If $\LEI_k(\theta,\alpha,c) \neq \varnothing$, $k' \geq k$, and $c' \leq c$, then also $\LEI_{k'}(\theta,\alpha,c')\neq\varnothing$.
        \item Continuity in $\theta$: If $\LEI_k(\theta,\alpha,c) \neq \varnothing$, then there exists an open interval $I$ containing $\theta$ and there exist $k' \geq k$ and $c' \leq c$ such that for all $\theta' \in I$, $\LEI_{k'}(\theta',\alpha,c') \neq \varnothing$.
    \end{enumerate}
    Monotonicity is an immediate consequence of \Cref{definition:set LEI}. To prove continuity, let $\bm\eta \in \LEI_k(\theta,\alpha,c)$. Let $\bm\lambda \in \{\bm\eta\} \times \Lambda^2$ be such that the closure of $Q(\bm\lambda)$ is contained in the interior of $Q(\bm\eta)$. By elementary geometry, $\bm\lambda \in \LEI_{k+2}(\theta',\alpha,\frac{c}{2})$ if $\theta'$ is sufficiently close to $\theta$. This proves continuity.
    
    Now we prove the lemma. By compactness of $\Theta$ and continuity, $\Theta$ can be covered by finitely many open intervals $I^{(j)}$ such that for each $j$, there exists $k^{(j)}, c^{(j)}$ such that for all $\theta \in I^{(j)}$, $\LEI_{k^{(j)}}(\theta,c^{(j)})$. By choosing $k = \max_j k^{(j)}$ and $c = \min_j c^{(j)}$ and using monotonicity, we have the lemma.
\end{proof}

\begin{proposition}\label{prop:expected-line-survival}
    Consider a non-degenerate grid model. For all angles $\theta\in [0,\pi]$ and all $n\in\N$, there is a set $E=E_{\theta,n}\subset\R$ of Lebesgue measure $\lesssim e^{-\Omega(n)}$, such that
    \begin{align}
        \label{eq:expected-line-survival}
        \P\bigl[\ell(\theta,t)\cap S_n \neq\varnothing\bigr]\lesssim \frac{1}{n} \qquad\text{for all }t\in \R\setminus E\,.
    \end{align}
    Consequently,
    \begin{align}
        \label{eq:expected-projection}
        \E|\proj_\theta S_n| \lesssim \frac{1}{n}\,.
    \end{align}
\end{proposition}

\begin{proof}
    Since the model is non-degenerate, by \Cref{lemma:nonempty V local} and \Cref{lem:non-deg-structure}, there exist $k,c$ such that for all $\theta \in [0,\pi]$, there exists $\alpha$ such that $\LEI_k(\theta,\alpha,c) \neq \varnothing$. This lemma then follows from \Cref{prop:expected-line-survival kc}, which now holds with estimates that are uniform in $\theta$.
\end{proof}

The proof of Theorem \ref{thm:expected-decay non-deg} is now immediate.
\begin{proof}[Proof of Theorem \ref{thm:expected-decay non-deg}]
    By Fubini,
    $
        \E[\Fav(S_n)]
        =
        \int_{0}^{\pi} \E |\proj_\theta S_n| \, d\theta
        \lesssim \frac{1}{n} \, .
    $

If the model is Ahlfors-regular, the lower bound $\E[\Fav(S_n)]\gtrsim\frac1n$ is a consequence of \eqref{Fav_lb_deterministic}. More generally, it is not hard to derive the lower bound from the fact $\P[ S_n\cap\ell\neq\varnothing]\gtrsim|\ell\cap[0,1)^2|/n$; see \Cref{lem:unif_ub_exp}.
\end{proof}

\begin{remark}
    The exceptional set in \Cref{prop:expected-line-survival} cannot be removed: Consider a random model which always picks exactly one of the squares in $\cD_1$ on the diagonal $\{Q(i,i) : i \in [L]\}$, uniformly at random. Now consider the line $\ell$ given by $y=x$. Then $\ell$ always intersects exactly one square of $S_n$, so  $\P[\ell \cap S_n \neq \varnothing] = 1$. The same is true for all lines that are ``$L^{-n}$-close'' to $\ell$. These lines must lie in the exceptional set of \Cref{prop:expected-line-survival}, and thus we do not have uniform decay over all lines. 

    On the other hand, there are many natural models, where the bound $\P[\ell\cap S_n\neq\varnothing]\lesssim\tfrac1n$ holds uniformly over all $n$ and over all lines. In particular, this is true if one selects the squares uniformly at random (i.e. $\P[S_1=S]=1/\binom{L^2}{L}$ for each $S$ that is a union of $L$ different squares in $\cD_1$). This also holds for the model used by Peres and Solomyak \cite{PeresSolomyak02}. The proof of this fact, which we omit here,  is based on an observation (a variant of the argument in \cite{PeresSolomyak02}), that in these models, the collection $\cB$ of the bad-squares  may be replaced by a sub-collection of $\cD_n$ that consists of only $o(L^{n})$ squares.
\end{remark}

\subsection{The degenerate case}

In \Cref{sec:non-deg case}, we only needed the set $\LEI_k(\theta,\alpha,c)$ to be nonempty. This is because there was a uniform choice of $k$ for all $\theta$. In the degenerate case, this is no longer true; $k$ must increase as $\theta$ approaches a degenerate direction. Thus, we need a more careful analysis of $\LEI_k(\theta,\alpha,c)$.

Define 
\begin{align}
    \label{eq:def W_k}
    W_{k+2} = \left\{ \mathbf{v}\mathbf{u}: \mathbf{u}\in\widetilde{\Lambda}^{k}\right\} \subset \Lambda^{k+2} \,,
    \qquad
    \text{where } \vv = \bigl((1,1),(L,1)\bigr), \widetilde{\Lambda}=\{L\} \times [L] \,.
\end{align}
(The notation $\mathbf{v}\mathbf{u}$ denotes the concatenation of $\mathbf{v}$ and $\mathbf{u}$.) The following is an analog of \Cref{lem:non-deg-structure} for degenerate models.

\begin{lemma} \label{lem:deg-structure}
    Suppose the grid model is vertically degenerate. If $k \in \N$ and $\tan\theta \in [L,L^{k}]$, then
    then \[W_{k+2} \subset \LEI_{k+2}(\theta,(1,1),\tfrac{1}{2}).\]
\end{lemma}
\begin{proof}
    Note that $\bigcup_{\mathbf{w}\in W_{k+2}} \overline{Q}(\mathbf{w})$ equals the rectangle
    \[  
        R_k := [L^{-1}-L^{-k-2},L^{-1}]\times [0,L^{-2}] \subset Q(1,1)\,.
    \]
    Any line with slope in $[0,L^{k}]$ that hits $R_k$ will hit the line $x=L^{-1}$ at a point with $y$-coordinate in $[0, 2L^{-2}]$. If the line additionally has slope $\ge L$, it will intersect all squares $\{Q(2,i) : i \in [L]\}$ in a segment of length at least $\frac{1}{2}L^{-1}$. Since the model is vertically degenerate, exactly one of these squares will be in $\cS_1$. This implies the lemma.
\end{proof}

The following combinatorial lemma is needed to obtain more precise bounds on the cardinality of the bad set $\cB$ used in the proof of \Cref{prop:expected-line-survival deg}, below. (The argument used to bound the cardinality of the bad set in the proof of \Cref{prop:expected-line-survival kc} is not enough.)

\begin{lemma}
    \label{lemma:substring event neg dep}
    Let $\lambda_1, \lambda_2,\ldots$ be independent random variables taking values in $\Lambda$. Let $A_m$ be the event $(\lambda_{m}, \ldots, \lambda_{m+k-1}) \in W_k$.
    Then for any finite $I \subset \N$,
    \[
        \P\left[\bigcap_{i\in I} A_i^c\right]
        \leq
        \prod_{i\in I}\P[A_i^c].
    \]
\end{lemma}

\begin{proof}
    Consider an arbitrary $J \subset [m-1]$.
    Observe that if $A_m$ occurs, then $\lambda_m = (1,1)$. Since $(1,1)$ appears in the strings in $W_k$ only in the first position (see \eqref{eq:def W_k}), we have
    \begin{align}
        \label{eq:A_m in complement}
        A_m \subset \bigcap_{i \in \{m-k+1, \ldots, m-1\}} A_i^c.
    \end{align}
    Thus, letting $J' = J \setminus \{m-k+1, \ldots, m-1\}$, we have
    \begin{align*}
        \P\left[A_m \cap \bigcap_{i \in J} A_i^c\right]
        \stackrel{\eqref{eq:A_m in complement}}{=}
        \P\left[A_m \cap \bigcap_{i \in J'} A_i^c\right]
        =
        \P\left[A_m\right] \P\left[\bigcap_{i \in J'} A_i^c\right]
        \geq
        \P\left[A_m\right] \P\left[\bigcap_{i \in J} A_i^c\right]
    \end{align*}
    Consequently, 
    \begin{align*}
        \P\left[A_m^c \cap \bigcap_{i \in J} A_i^c\right]
        =
        \P\left[\bigcap_{i \in J} A_i^c\right] - \P\left[A_m \cap \bigcap_{i \in J} A_i^c\right]
        \leq
        \P\left[A_m^c\right] \P\left[\bigcap_{i \in J} A_i^c\right]
    \end{align*}
    The lemma follows by repeatedly applying this inequality.
\end{proof}

\begin{proposition}\label{prop:expected-line-survival deg}
Consider a vertically degenerate grid model. For all $k\geq 2$ and all angles $\theta$ with $|\tan\theta| \in [L, L^k]$ and all $n \geq L^k$, there is a set $E=E_{\theta,k,n}\subset\R$ of Lebesgue measure $\lesssim e^{-\Omega(n/L^k)}$, such that
\begin{align}
    \label{eq:expected-line-survival deg}
    \P\bigl[\ell(\theta,t)\cap S_n \neq\varnothing\bigr]\lesssim \frac{L^k}{n} \,,\quad\text{for all }t\in \R\setminus E\,.
\end{align}
Consequently,
\begin{align}
    \label{eq:expected-projection deg}
    \E|\proj_\theta S_n| \lesssim \frac{L^k}{n} \,.
\end{align}
\end{proposition}

\begin{proof}
    Fix $\theta$ with $|\tan\theta| \in [L,L^k]$. We may assume without loss of generality that $\tan\theta > 0$. (The $\tan\theta < 0$ case follows from symmetry.) Let
    \begin{align*}
        \cG &= \left\{ \bm \lambda\in\Lambda^n: \Count(\bm\lambda, \LEI_{k+2}(\theta,(1,1),\tfrac{1}{2})) \ge \tfrac{n-k-1}{2L^k} \right\}
        \\
        \cB &= \cD_n \setminus \cG.
    \end{align*}
    By \Cref{lemma:expect multi scale}, for all $t \in \R$,
    \[
        \P\bigl[\ell(\theta,t) \cap (\cup \cG) \cap S_n\neq\varnothing\bigr]\lesssim %
        \frac{L^{k}}{n-k-1} \sim \frac{L^{k}}{n}.
    \]

    Next, by \Cref{lem:deg-structure}, $\LEI_{k+2}(\theta,(1,1),\frac{1}{2}) \supset W_{k+2}$, so
    \[
        \cB 
        \subset
        \left\{ \lambda\in\Lambda^n: \Count(\bm\lambda, W_{k+2}) \le \tfrac{n-k-1}{2L^k}  \right\}
        \,.
    \]
    Recall that 
    \[
    \Count(\bm\lambda, W_{k+2}) = \#\{i\in \{0, \ldots,n-k-2\}:(\lambda_{i+1},\ldots,\lambda_{i+k+2}) \in W_{k+2} \} \,.
    \]
    Note that $\# W_{k+2} = L^{k}$.  By \Cref{lemma:substring event neg dep}, the indicators of the events $(\lambda_{i+1},\ldots,\lambda_{i+k+2}) \in W_{k+2}$, $i\in \{0, \ldots,n-k-2\}$ are $\textup{Bernoulli}(L^{-k})$ random variables which are jointly negatively correlated in the sense of \eqref{eq:neg dep}. Thus, we can apply the Chernoff bound (\Cref{lemma:chernoff neg dep}) to get
    \[
        \#\mathcal{B} \le \exp\left(-\frac{n-k-1}{16 L^{k}}\right)\cdot  \# \Lambda^n
        \le
        \exp\left(-\frac{n}{64 L^{k}}\right)\cdot  \# \Lambda^n.
    \]
    By \Cref{lemma:markov}, there is a set $E = E_{\theta,k,n} \subset\R$ of Lebesgue measure $\lesssim e^{-\Omega(n/L^k)}$, such that
    \[ 
        \P\bigl[\ell(\theta,t)\cap (\cup \cB) \cap S_n \neq\varnothing\bigr] \lesssim \frac{L^k}{n} \quad\text{for all }t\in \R\setminus E.
    \]
    This proves \eqref{eq:expected-line-survival deg}. To deduce \eqref{eq:expected-projection deg}, we observe
    \begin{align*}
        \E |\proj_\theta S_n| 
        &\leq
        \int_{\R \setminus E_{\theta,n}} \P[\ell(\theta,t) \cap S_n\neq\varnothing] \, dt
        +
        |E_{\theta,k,n}|
        \\
        &\lesssim
        \frac{L^k}{n}
        +
        e^{-\Omega(n/L^k)}
        \lesssim
        \frac{L^k}{n}\,.        
    \end{align*}
\end{proof}

\begin{proof}[Proof of Theorem \ref{thm:expected-decay deg}]
    The lower bound $\E[\Fav(S_n)]\gtrsim\frac{\log n}{n}$ is derived in \Cref{sec:bateman-volberg}. Towards the upper bound, let us first consider a grid model that is vertically degenerate but not horizontally degenerate. 
    
    Fix $n \geq 4$. We cover $[0,\pi]$ by the following intervals:
    \begin{align*}
        \Theta_k &= \{\theta \in [0,\pi] : |\tan \theta| \in [L^k, L^{k+1}]\}, \qquad k = 1, \ldots, \lfloor \log_L n \rfloor
        \\
        \Theta_{\textup{vert}} &= \{\theta \in [0,\pi] : |\tan\theta| \in [n, \infty]\}
        \\
        \Theta_{\textup{nondeg}} &= \{\theta \in [0,\pi] : |\tan\theta| \in [0, L]\}
    \end{align*}

    For $k\in[\lfloor \log_L n \rfloor]$, we use the  Proposition \ref{prop:expected-line-survival deg} and note that $|\Theta_k| \sim L^{-k}$ to conclude 
    \begin{align*}
        \int_{\Theta_k} \E |\proj_\theta S_n| \, d\theta
        \lesssim |\Theta_k| \frac{L^k}{n}
        \sim \frac{1}{n} 
        \,.
    \end{align*}
    For $\Theta_{\textup{vert}}$, we note that $|\Theta_{\textup{vert}}| \sim \frac{1}{n}$ and use the trivial bound
    \begin{align*}
        \int_{\Theta_{\textup{vert}}} \E |\proj_\theta S_n| \, d\theta
        \lesssim |\Theta_{\textup{vert}}|
        \sim \frac{1}{n}
        \,.
    \end{align*}
    For $\Theta_{\textup{nondeg}}$, we use \Cref{lemma:nonempty V local}, \Cref{lem:non-deg-structure}, and \Cref{prop:expected-line-survival kc} to obtain 
    \[
        \E|\proj_\theta S_n| \lesssim \frac{1}{n} \qquad\text{for all } \theta \in \Theta_{\textup{nondeg}} \,.
    \]
    By combining the three estimates together,
    \begin{align*}
        \E[\Fav(S_n)] = \int_{0}^{\pi} \E  |\proj_\theta S_n| \, d\theta
        &\leq
        \int_{\Theta_{\textup{vert}}} 
        +
        \int_{\Theta_{\textup{nondeg}}}
        +
        \sum_{k=1}^{\lfloor \log_L n \rfloor}
        \int_{\Theta_k}
        \\
        &\lesssim
        \frac{1}{n}
        +
        \frac{1}{n}
        +
        \sum_{k=1}^{\lfloor \log_L n \rfloor}
        \frac{1}{n}
        \sim \frac{\log n}{n}.
    \end{align*}

    This proves the theorem when the grid model is vertically but not horizontally degenerate. If the grid model is horizontally but not vertically degenerate, the argument is the same, by symmetry. If the grid model is both vertically and horizontally degenerate, we repeat the argument but with the following sets instead:
    \begin{align*}
        \Theta_k &= \{\theta \in [0,\pi] : |\tan \theta| \in [L^k, L^{k+1}]\}, &k = 1, \ldots, \lfloor \log_L n \rfloor
        \\
        \widetilde \Theta_{k} &= \{\theta \in [0,\pi] : |\tan\theta| \in [L^{-k-1}, L^{-k}]\}, &k = 1, \ldots, \lfloor \log_L n \rfloor
        \\
        \Theta_{\textup{vert}} &= \{\theta \in [0,\pi] : |\tan\theta| \in [n, \infty]\}
        \\
        \Theta_{\textup{horiz}} &= \{\theta \in [0,\pi] : |\tan\theta| \in [0, \tfrac{1}{n}]\}
        \\
        \Theta_{\textup{nondeg}} &= \{\theta \in [0,\pi] : |\tan\theta| \in [L^{-1}, L]\}
    \end{align*}
    The integral of $\E |\proj_\theta S_n|$ over each of the $2\lfloor \log_L n \rfloor + 3$ sets above is $\lesssim \frac{1}{n}$, so $\E[\Fav(S_n)] \lesssim \frac{\log n}{n}$.
\end{proof}

\section{The sharp asymptotic value of the Favard length}\label{sec:sharp_asymptotic}

Our goal in this section is to prove the following.

\begin{theorem}\label{thm:limit}
    For a non-degenerate uniform grid model, define
    \[
        V(\theta)=\int_\R \Var\bigl(L|\ell(\theta,t)\cap S_1|\bigr) \, dt\,.
    \]
    Then
    \[
        \lim_{n\to\infty} n\E|\proj_\theta S_n| = \frac{2}{V(\theta)} \qquad\text{uniformly in } \theta \, ,
    \]
    and, therefore,
    \[
        \lim_{n\to\infty} n\E[\Fav(S_n)] = \int_{0}^{\pi}\frac{2}{V(\theta)}\,d\theta\,.
    \]
    Consequently, by \Cref{thm:almost-sure-convergence}, almost surely
    \[
        \lim_{n \to \infty} n\Fav(S_n) =  Z\int_{0}^{\pi}\frac{2}{V(\theta)}\,d\theta\,.
    \]
\end{theorem}

\begin{corollary}\label{cor:nghbrhd_Favard_length}
    For a non-degenerate uniform surviving grid model,
    \begin{equation*}
\lim_{r\downarrow 0}\log_L(1/r)\Fav(S(r))=Z\int_{0}^{\pi}\frac{2}{V(\theta)}\,d\theta\qquad\text{almost surely}.
\end{equation*}
\end{corollary}
\begin{proof}
    By monotonicity of $\Fav(S(r))$ in $r$ and the fact $\tfrac{\log_L(L^{n+1})}{\log_L(L^n)}=\tfrac{n+1}{n}\longrightarrow1$, we may take the limit along $r=L^{-n}$. By Lemma \ref{lem:step-vs-neighbourhood}, there is a constant $C<\infty$ such that
    \[
         n\Fav(S_{n+C}) \le n\Fav(S(L^{-n})) \le (n+C)\Fav(S_{\lfloor n-\log n\rfloor})\,. 
    \]
  Noting that $\lim_{n\to \infty}\tfrac{n}{n+C}=\lim_{n\to \infty}\tfrac{\lfloor n-\log n\rfloor}{n}=1$ and applying \Cref{thm:limit}, it follows that
    \begin{align*}
       I= \lim_{n\to \infty}n\Fav(S_{n+C})\le\lim_{n\to \infty} n\Fav(S(L^{-n}))\le\lim_{n\to \infty} n\Fav(S_{\lfloor n-\log n\rfloor})=I\,,
    \end{align*}
    where $I=Z\int_{0}^{\pi}\frac{2}{V(\theta)}\,d\theta$.
\end{proof}

Recall \Cref{remark:half-open}: squares in this section are half-open.

\subsection{Preliminary steps and lemmas}

\subsubsection{Overview of proof}

Our proof of \Cref{thm:limit} involves a refinement of the arguments in Lemmas \ref{lemma:expect single scale} and \ref{lemma:expect multi scale}. We start by recalling \eqref{eq:length-w-square-v1}, i.e., for all $n$,
\begin{align}\label{eq:length-w-square}
|\ell(\theta,t) \cap [0,1)^2|
=
\E\bigl[ L^n |\ell(\theta,t) \cap S_n| \bigm| \ell(\theta,t) \cap S_n \neq \varnothing\bigr] \cdot \P[\ell(\theta,t) \cap S_n \neq \varnothing]\,.
\end{align}
(Recall \Cref{remark:half-open}: squares in this section are half-open.) We are therefore interested in estimating
\begin{align}
    \frac{1}{n} \E\bigl[ L^n |\ell(\theta,t) \cap S_n| \bigm| \ell(\theta,t) \cap S_n \neq \varnothing\bigr]\,.
\end{align}

We fix a total order on $\Lambda$. (Unlike in the proof of \Cref{lemma:expect multi scale}, this total order can be completely arbitrary.) For each $n$, this yields a lexicographic order on $\Lambda^n$ and hence on $\cD_n$ (all denoted by $\prec$). For fixed $\theta$, let $Q_{n,t}$ be the $\cD_n$-valued random variable defined by 
\begin{align}
    \label{eq:def Q_n,t}
    Q_{n,t}
    =
    \min\{Q \in \cS_n : \ell(\theta,t) \cap Q \neq \varnothing\}
    \,,
\end{align}
assuming that $\ell(\theta,t)\cap S_n\neq\varnothing$. For completeness, we define $Q_{n,t}=\varnothing$ if $\ell(\theta,t)\cap S_n=\varnothing$.

We observe that the event $\ell(\theta,t) \cap S_n\neq\varnothing$ is the disjoint union of the events $\{Q_{n,t} = Q(\bm\lambda) : \bm\lambda \in \Lambda^n\}$. Therefore, we would like to estimate
\begin{align}
\label{eq:1/n E conditioned on lambda}
    \frac{1}{n} \E\bigl[ L^n |\ell(\theta,t) \cap S_n| \bigm| Q_{n,t} = Q(\bm\lambda) \bigr]
\end{align}
for all $\bm\lambda \in \Lambda^n$, but it turns out there are two small exceptional subsets of $\Lambda^n$ (a ``geometric'' one and a ``combinatorial'' one; see \Cref{sec:exceptional sets} below) where we cannot get a good control. If $\bm\lambda$ is outside of these two exceptional sets, then we can use several approximation lemmas (proved in \Cref{section:approximation lemmas} below) to show that \eqref{eq:1/n E conditioned on lambda} is close to $\frac{V(\theta)}{2}$; this is done in \Cref{prop:expectation_conditional_on_approx_distr}.

\begin{remark}
Some parts of the proof of \Cref{thm:limit} are inspired by the proof of Kolmogorov's estimate for critical Galton--Watson processes, mentioned in \Cref{section:intro ideas of proof}. For example, \eqref{eq:length-w-square} and \Cref{lemma:decomp X_ni} combined can be viewed as an analogue of  \cite[equation (12.9)]{LyonsPeres16}, and \Cref{lemma:probability given first survivor} bears a similarity to  \cite[Exercise 12.8]{LyonsPeres16}.
\end{remark}

Next, we prove several lemmas that will be needed in the proof of Theorem \ref{thm:limit}.

\subsubsection{Properties of the variance $V(\theta)$}

\begin{lemma}[Non-degeneracy]\label{lem:V_lb}
  There is $c>0$ (depending on the model) such that  $V(\theta)>c$ for all $\theta$.
\end{lemma}
\begin{proof}
One can check that $V$ is continuous in $\theta$, and therefore by compactness it is enough to show that $V(\theta)>0$ for all $\theta$. For $\theta\in \{0,\pi/2\}$, this follows from the non-degeneracy of the model. For other values of $\theta$, note that there is a positive measure set of $t$ such that $\ell(\theta,t)$ intersects a single square in $\cD_1$, and so the random variable $|\ell(\theta,t) \cap S_1|$ is not almost surely constant, showing that $V(\theta)>0$ in this case also.
\end{proof}

\begin{lemma}[Alternative expression for $V(\theta)$]\label{lem:V(theta) alternative} 
For each fixed $\theta$,
\begin{align}
\label{eq:V(theta) alternative}
L^{-2} V(\theta)
=
\sum_{Q_1, Q_2 \in \cD_1, Q_1 \neq Q_2}
\P[Q_1 \in \cS_1 \land Q_2 \in \cS_1]
\int_\R
|\ell(\theta,t) \cap Q_1| |\ell(\theta,t) \cap Q_2| 
\, dt\,.
\end{align}
\end{lemma}
\begin{proof}
Since $|\ell(\theta,t) \cap S_1| = \sum_{Q \in \cD_1} |\ell(\theta,t) \cap Q| \bbone[Q \in \cS_1]$, we have
\begin{align*}
\int_\R \E |\ell(\theta,t) \cap S_1|^2 \, dt
&=
\sum_{Q_1, Q_2 \in \cD_1}   \P[Q_1 \in \cS_1 \land Q_2 \in \cS_1] \int_\R |\ell(\theta,t) \cap Q_1| |\ell(\theta,t) \cap Q_2| \, dt \,.
\end{align*}
Note that the sum of the off-diagonal terms ($Q_1 \neq Q_2$) is precisely the right-hand side of \eqref{eq:V(theta) alternative}. 
To handle the diagonal terms, first notice that, by rescaling, for any $Q\in\cD_1$ we have
\[
\int_\R |\ell(\theta,t) \cap Q|^2 \, dt = \frac{1}{L^3} \int_\R |\ell(\theta,t) \cap [0,1)^2|^2 \, dt.
\]
Hence, the sum of the terms $Q_1 = Q_2$ is
\begin{align*}
\sum_{Q \in \cD_1}   \P[Q \in \cS_1] \int_\R |\ell(\theta,t) \cap Q|^2 \, dt
&=
\sum_{Q \in \cD_1}   \frac{1}{L} \cdot \frac{1}{L^3} \int_\R |\ell(\theta,t) \cap [0,1)^2|^2 \, dt
\\
&=
\frac{1}{L^2}
\int_\R |\ell(\theta, t) \cap [0,1)^2|^2 \, dt
\\
&=
\int_\R (\E |\ell(\theta,t) \cap S_1|)^2 \, dt \,.
\end{align*}
Combining the expressions for the diagonal and off-diagonal contributions, we get \eqref{eq:V(theta) alternative}.
\end{proof}

\subsubsection{Two approximation lemmas}
\label{section:approximation lemmas}

From now on, the direction $\theta$ will be fixed until nearly the end of the proof of Theorem \ref{thm:limit}. We therefore suppress the dependence on $\theta$ from much of the notation. We emphasize, however, that implicit constants do \emph{not} depend on $\theta$.

The first approximation lemma says that for certain events, conditioning on $Q_{n,t} = Q(\bm\eta)$ is not too different from conditioning on $Q(\eta_1) \in \cS_1$.

\begin{lemma}[Conditioning on $Q_{n,t} = Q(\bm\eta)$ versus on $Q(\eta_1) \in \cS_1$]
\label{lemma:probability given first survivor}
For every $n,\theta$, there exists a set $E_{n} \subset \R$ of size $O(\exp(-\Omega(n)))$, which is a union of intervals of length $\ge L^{-n}$, such that the following holds: for any $\bm\eta \in \Lambda^n$ and $\kappa \in \Lambda \setminus \{\eta_1\}$ such that 
\begin{align} \label{eq:non-trivial-Q1Q2}
\P\bigl[Q(\kappa) \in\cS_1 \land Q(\eta_1) \in \cS_1\bigr] > 0,
\end{align}
we have
\[
\left|\frac{\P\bigl[Q(\kappa) \in \cS_1 \bigm| Q_{n,t} = Q(\bm\eta) \bigr]}{\P\bigl[Q(\kappa) \in \cS_1 \bigm| Q(\eta_1) \in \cS_1\bigr]} -1\right| = O(1/n)\qquad\text{for all } t \in \R \setminus E_{n}\,.
\]
\end{lemma}

\begin{proof}
Given $\eta\in\Lambda$, let $X(\eta)$ denote the event $Q(\eta)\in \cS_1$. Let 
\[\Gamma=\{(\kappa,\alpha)\in\Lambda^2 :  \P[X(\kappa)\land X(\alpha)]>0\}
\qquad\text{and}\qquad
c=\min_{(\kappa,\alpha)\in\Gamma\,}\P[X(\kappa)\land X(\alpha)]\,.\]
Note that $c>0$ is independent of $n$ and it only depends on the model.

By \Cref{prop:expected-line-survival} (see also Lemma \ref{lemma:markov}), there exists a set $E_{n} \subset \R$ of size $O(\exp(-\Omega(n)))$, which is a union of intervals of length $\ge L^{-n}$, such that
\begin{align}
\label{eq:E_n complement}
\R \setminus E_n
\subset 
\left\{t \in \R : \P\bigl[\ell(\theta,t) \cap S_n \neq \varnothing\bigr] \lesssim \frac{1}{n}\right\}\,,
\end{align}
provided the implicit constant in \eqref{eq:E_n complement} is chosen appropriately.

Fix $\bm\eta \in \Lambda^n$ and $\kappa \in \Lambda \setminus\{\eta_1\}$ such that \eqref{eq:non-trivial-Q1Q2} holds. By independence, the process inside $Q(\eta_1)$ is independent of the event $Q(\kappa) \in \cS_1$. Thus, we have
\begin{align*}
&\P\bigl[Q(\kappa) \in \cS_1 \bigm| Q_{n,t} = Q(\bm\eta) \bigr]
=
\P\bigl[ X(\kappa) \bigm| X(\eta_1) \land A\bigr],
\end{align*}
where $A$ is the event
\[
\ell(\theta,t) \cap Q \cap S_n = \varnothing \qquad \text{for all $Q\in \cD_1$ with $Q\prec Q(\eta_1)$}.
\]
Thus, to prove the lemma, it suffices to show
\begin{align}
\label{eq:prob given first survivor rewritten}
\left|\frac{\P\bigl[ X(\kappa) \bigm| X(\eta_1) \land A\bigr]}{\P\bigl[ X(\kappa) \bigm| X(\eta_1)\bigr]} -1\right| = O(1/n)\qquad\text{for all } t \in \R \setminus E_{n}\,.
\end{align}

If $t \in \R \setminus E_{n}$ then,
\begin{align}
\P[A^c]
\leq
\P\bigl[\ell(\theta,t) \cap S_n \neq \varnothing\bigr]
\stackrel{\eqref{eq:E_n complement}}{\lesssim}
\frac{1}{n}\,.
\end{align}
Since \eqref{eq:non-trivial-Q1Q2} implies $(\kappa,\eta_1)\in \Gamma$, we have \[\P[X(\kappa)\land X(\eta_1)]\ge c\,.\] 
Writing
\[
\frac{\P[X(\kappa) \mid X(\eta_1) \land A]}{\P[X(\kappa) \mid X(\eta_1)]} = \frac{\P[A \mid X(\kappa)\land X(\eta_1)]}{\P[A \mid X(\eta_1)]}\,
\]
and recalling that $\P[A^c]\lesssim 1/n\le 1 \lesssim \P[X(\kappa)\land X(\eta_1)]\le \P[X(\eta_1)]$, we get \eqref{eq:prob given first survivor rewritten} and hence the lemma.
\end{proof}

The second approximation lemma is a purely geometric observation. It allows us to approximate expressions of the form $|\ell(\theta,t) \cap Q(\kappa)|$ by ``pieces'' of the integrals appearing in \eqref{eq:V(theta) alternative}.

\begin{lemma}[Lengths of intersections of lines with squares]
\label{lem:length_to_integral}
    For $\theta\in [0,\pi]$, $\bm\eta \in \Lambda^k$, and $\kappa\in\Lambda$, define 
    \begin{align}
    \begin{split}
        \label{eq:def V(theta,eta,kappa) and E}
        V(\theta,\bm\eta,\kappa) &= 
        L^{2k} \int_{\R}  |\ell(\theta, s) \cap Q(\bm\eta)| |\ell(\theta, s) \cap Q(\kappa)| \, ds \\
        \operatorname{Err}(\theta,\bm\eta,\kappa)
        &=
        \sup_{\substack{t \in \R \\ \ell(\theta,t) \cap Q(\bm\eta) \neq \varnothing}}
        \bigl|
        \left|\ell(\theta,t) \cap Q(\kappa) \right|
        -
        V(\theta,\bm\eta,\kappa)
        \bigr| \,.
    \end{split}
    \end{align}
    Then for any $\theta\in [0,\pi]$ and $\kappa \in \Lambda$,
    \begin{align}
    \label{eq:length-to-integral-sum}
        \frac{1}{\#\Lambda^k}\sum_{\bm\eta \in \Lambda^k} 
        \operatorname{Err}(\theta,\bm\eta,\kappa)
        \lesssim
        \frac{1}{L^k}
    \end{align}    
\end{lemma}

\begin{proof}
    Fix $\kappa \in \Lambda$ and let $g(t)=|\ell(\theta,t)\cap Q(\kappa)|$. For $\bm\eta \in \Lambda^k$, note that $L^{2k} \int_{\R}  |\ell(\theta, t) \cap Q(\bm\eta)|  \, dt = L^{2k} \cL^2(Q(\bm\eta)) = 1$, so we can view $V(\theta,\bm\eta,\kappa)$ as a weighted average of $g$ on $I_{\bm\eta} := \proj_\theta(Q(\bm\eta)) = \{t \in \R : \ell(\theta,t) \cap Q(\bm\eta) \neq \emp\}$. Thus,
    \begin{align*}
        \Err(\theta,\bm\eta,\kappa)
        &=
        \sup_{t \in I_{\bm\eta}}
        \left|
        g(t)
        -
        V(\theta,\bm\eta,\kappa)
        \right|
        \leq
        \sup_{s,t \in I_{\bm\eta}}  |g(s) - g(t)|
        \,.
    \end{align*}
    We make the following elementary geometry observations: $I_{\bm\eta}$ is an interval of length $\lesssim L^{-k}$, $g$ is $|\sin\theta\cos\theta|^{-1}$-Lipschitz, and $\|g\|_\infty \lesssim 1$. These give us 
    \begin{align*}
        \Err(\theta,\bm\eta,\kappa)
        \lesssim
        \min\left\{
        \frac{1}{L^k|\sin\theta \cos\theta|}, 1
        \right\}
        \, .
    \end{align*}
    Summing over $\bm\eta \in \Lambda^k$,
    \begin{align}
    \label{eq:sum Err}
        \sum_{\bm\eta \in \Lambda^k}\Err(\theta,\bm\eta,\kappa)
        \lesssim
        \min\left\{
        \frac{1}{L^k|\sin\theta \cos\theta|}, 1
        \right\}
        \cdot
        \#\{ \bm\eta \in \Lambda^k : \Err(\theta,\bm\eta,\kappa) > 0 \}
        \, .
    \end{align}
    To bound the cardinality above, we make another elementary geometry observation: $g$ is a piecewise linear function that is constant except on a set $J \subset \R$ which is the union of at most two intervals of length $O(|\sin\theta\cos\theta|)$. If $\Err(\theta,\bm\eta,\kappa) > 0$, then $I_{\bm\eta} \cap J \neq \emp$, so $Q(\bm\eta) \subset \proj_\theta^{-1}(J(\sqrt{2} L^{-k}))$. (See \Cref{figure:approx intersection}.) An area-counting argument then yields
    \begin{align*}
        \#\{ \bm\eta \in \Lambda^k : \Err(\theta,\bm\eta,\kappa) > 0\} 
        &\leq
        \#\{ \bm\eta \in \Lambda^k : Q(\bm\eta) \subset \proj_\theta^{-1}(J(\sqrt{2} L^{-k})) \cap [0,1)^2\} 
        \\
        &\lesssim 
        L^{2k}(|\sin\theta\cos\theta| + L^{-k}) \,.
    \end{align*}
    Combining this with \eqref{eq:sum Err} yields the lemma.
\end{proof}

\begin{figure}[h]
\centering
\includegraphics[width=0.5\textwidth]{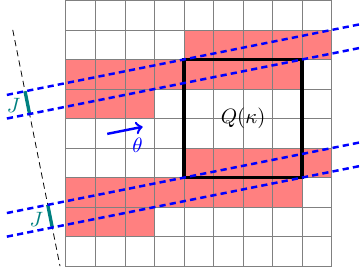}
\caption{This illustrates the last part of the proof of \Cref{lem:length_to_integral}. In the displayed grid, the shaded squares are precisely those $Q(\bm\eta)$ for which $\Err(\theta,\bm\eta,\kappa) > 0$.}
\label{figure:approx intersection}
\end{figure}

\subsubsection{Exceptional sets of squares}
\label{sec:exceptional sets}

In the hypothesis of \Cref{prop:expectation_conditional_on_approx_distr} below, there will be two exceptional sets of squares. The first one is a ``combinatorial'' one, related to the uniformity of substrings of $\bm\lambda \in \Lambda^n$. Recall (from \Cref{def:Count}) that for $\bm\lambda \in \Lambda^n$ and $\bm\eta \in \Lambda^k$,
\[
    \Count(\bm\lambda, \bm\eta)
    =
    \#\{i\in \{0,\ldots, n-k\} \,:\,(\lambda_{i+1},\ldots,\lambda_{i+k})= \bm\eta\}\,.
\]
We define the following.
\begin{definition}
Given $1 \leq k < n$ and $\epsilon>0$, we say that $\bm\lambda\in\Lambda^n$ is \emph{$(k,\varepsilon)$-approximately uniform}, if
\begin{align}
    \label{eq:approx-uniform}
    \left|\Count(\bm\lambda, \bm\eta) - (n-k+1)L^{-2k} \right|
    \leq
    \epsilon (n-k+1)L^{-2k}\,
    \qquad\text{for all $\bm\eta \in \Lambda^k$.}
\end{align}
\end{definition}

The $\bm\lambda \in \Lambda^n$ which are not $(k,\varepsilon)$-approximately uniform will form our combinatorial exceptional set. We can bound the cardinality of this set using the following.

\begin{lemma}[Bound using Chernoff and partitioning into $k$-blocks]\label{lem:chern_k_blocks}
    Let $n \geq 4k$, $\bm\eta \in \Lambda^k$, and $0<\varepsilon<\tfrac14$. Define
    \begin{align*}
        \cB &= \left\{ \bm\lambda\in\Lambda^{n}: \left|\Count(\bm\lambda, \bm\eta) - (n-k+1) L^{-2k}\right| \geq \epsilon (n-k+1) L^{-2k}\right\} \, .
    \end{align*}
    Then $\# \cB =O_k(1)\exp(-\Omega_{k,\varepsilon}(n))\cdot\#\Lambda^n$.
\end{lemma}

\begin{proof}
    Fix $0\le j\le k-1$. Let $l_j=\lfloor \frac{n-j}{k} \rfloor$ and set
    \[
    \cB_j=\left \{\bm\lambda\in\Lambda^{n}: |\#\{m\in[l_j]\,:\,(\lambda_{(m-1)k+j+1},\ldots,\lambda_{mk+j}) = \bm\eta \} - l_j L^{-2k}| \geq \epsilon l_jL^{-2k}\right\}
    \]
    By Chernoff's inequality (\Cref{lemma:chernoff}),
    \[\frac{\#\cB_j}{\#\Lambda^n}\le 2 \exp\left(-\frac{\varepsilon^2 l_j}{4 L^{2k}}\right)\le 2 \exp\left(-\frac{\varepsilon^2 n}{8 k L^{2k}}\right)\,.\]
    The claim follows from this noting that $\cB\subset\bigcup_{j=0}^{k-1}\cB_j$.
\end{proof}

Next, we turn to the ``geometric'' exceptional set of squares, which will arise from \Cref{lemma:probability given first survivor}. To define it, we need to introduce additional notation. Recall that $h_Q: [0,1)^2\to Q$ is the natural homothety. For $0\le i\le n$ and for any $t\in\R$, $\bm\lambda\in\Lambda^n$, let $s(t,\bm\lambda,i)$ be defined by 
\begin{align}
    \label{eq:def s(t,lambda,i)}
    \ell(\theta,s(t,\bm\lambda,i)) = h^{-1}_{Q(\bm\lambda^{(i)})}(\ell(\theta,t)).
\end{align}
Define
\[
    \cH_{n,i}=\left\{\bm\lambda\in\Lambda^n\,:\,\text{there exists a }t \in \R \text{ such that } \ell(\theta,t)\cap Q(\bm\lambda)\neq\varnothing \text{ and } s(t,\bm\lambda,i)\in E_{n-i} \right\}\,,
\]
where $E_{n-i}=E_{n-i,\theta}$ is as in Lemma \ref{lemma:probability given first survivor}. 
\begin{lemma}\label{lem:H_n_size}
    $\#\cH_{n,i}\lesssim\exp(-\Omega(n-i))\# \Lambda^n$.
\end{lemma}

\begin{proof}
Recall that the set $E_{n-i}$ from Lemma \ref{lemma:probability given first survivor} is a union of intervals of length $\ge L^{i-n}$, and satisfies $|E_{n-i}|\lesssim\exp(-\Omega(n-i))$. Let 
\[
    \cB_{n-i}=\{\bm\lambda\in\Lambda^{n-i}\mid E_{n-i}\cap \proj_\theta Q(\bm\lambda)\neq\varnothing\}\,.
\] 
Then,
\[
\left|\bigcup_{\bm\lambda\in\cB_{n-i}}Q(\bm\lambda)\right|\lesssim\exp(-\Omega(n-i))\,.
\]
It follows that $\#\cB_{n-i}=O(\exp(-\Omega(n-i)))\cdot \#\Lambda^{n-i}$. Noting that 
\[
    \cH_{n,i}=\bigcup_{\bm\eta\in\Lambda^{i}}\{\bm\eta\bm\lambda\mid \bm\lambda\in \cB_{n-i}\}\,,
\]
we obtain the claim. (Recall that $\bm\eta\bm\lambda$ denotes concatenation.)
\end{proof}

For the exceptional squares arising from the two lemmas above, their contribution will be negligible in the proof of \Cref{thm:limit}. To control them, we use the following lemma, which only takes into account the cardinality of $\cB$.

\begin{lemma}
\label{lemma:integral Qnt in B}
For any $\cB \subset \Lambda^n$,
\begin{align}
    \int_{\R} \P[Q_{n,t} \in \cB] \, dt
    \lesssim \frac{\#\cB}{\#\Lambda^n}
    \,.
\end{align}
(Note: In the statement ``$Q_{n,t} \in \cB$,'' we slightly abuse notation and identify $\cD_n$ with $\Lambda^n$.)
\end{lemma}

\begin{proof}
We have
\begin{align*}
    \P[Q_{n,t} \in \cB]
    &=
    \sum_{\bm\lambda \in \cB} \P[Q_{n,t} = Q(\bm\lambda)] \bbone[\ell(\theta,t) \cap Q(\bm\lambda) \neq \emp]
    \\
    &\leq
    \sum_{\bm\lambda \in \cB} \P[Q(\bm\lambda) \in \cS_n] \bbone[\ell(\theta,t) \cap Q(\bm\lambda) \neq \emp]
    \\
    &\stackrel{\eqref{eq:squarewise-uniform}}{=}
    L^{-n} \sum_{\bm\lambda \in \cB}  \bbone[\ell(\theta,t) \cap Q(\bm\lambda) \neq \emp]
    \, .
\end{align*}
Integrating in $t$ and using Fubini gives the desired result.
\end{proof}

\subsection{Proof of \Cref{thm:limit}}

\begin{lemma}[Decomposition into scales]
\label{lemma:decomp X_ni}
For every $\bm\lambda \in \Lambda^n$,
\begin{align}\label{eq:reduct-to-average step1}
\begin{split}
&\E\bigl[L^n |\ell(\theta,t) \cap S_n| \bigm| Q_{n,t} = Q(\bm\lambda) \bigr]
\\
&=
L^n |\ell(\theta,t)\cap Q(\bm\lambda)| 
+
\sum_{i=0}^{n-1} \E \bigl[X_{n,i} \bigm| Q_{n,t} = Q(\bm\lambda)  \bigr]
\,.
\end{split}
\end{align}
where
\begin{align*}
    X_{n,i}
    = 
    \sum_{\kappa\in\Lambda, \kappa \succ \lambda_{i+1}}
    L^{i+1} |\ell(\theta,t) \cap S_{i+1} \cap Q(\bm\lambda^{(i)}\kappa)|\,.
\end{align*}
\end{lemma}

\begin{proof}
Observe that
\begin{align}\label{eq:def_X'}
L^n |\ell(\theta,t) \cap S_n| 
=
L^n |\ell(\theta,t) \cap S_n \cap Q(\bm\lambda)|
+
\sum_{i=0}^{n-1} X_{n,i}'
\end{align}
where
\begin{align*}
    X_{n,i}' 
    &= 
    L^n |\ell(\theta,t) \cap S_n \cap Q(\bm\lambda^{(i)}) \setminus Q(\bm\lambda^{(i+1)})|
    \\
    &=  
    \sum_{\kappa\in\Lambda\setminus\{\lambda_{i+1}\}} L^n |\ell(\theta,t) \cap S_n \cap Q(\bm\lambda^{(i)}\kappa)|\,.
\end{align*}
We claim that, for each $n,i$,
\begin{align} \label{eq:equal-conditional-expectations}
    \E \bigl[X_{n,i}' \bigm| Q_{n,t} = Q(\bm\lambda) \bigr] 
    = 
    \E \bigl[X_{n,i} \bigm| Q_{n,t} = Q(\bm\lambda) \bigr]
    \,.
\end{align}
To see this, first note that by definition of $Q_{n,t}$ (see \eqref{eq:def Q_n,t}), 
\[
\E \bigl[ L^n|\ell(\theta,t)\cap S_n\cap Q(\bm\lambda^{(i)}\kappa) | \bigm| Q_{n,t} = Q(\bm\lambda) \bigr]
=
0
\qquad\text{if }\kappa \prec \lambda_{i+1} \,.
\]
Thus, to prove \eqref{eq:equal-conditional-expectations}, it is enough to show that if $\kappa \succ \lambda_{i+1}$, then
\begin{align*}
    &\E \bigl[ L^n|\ell(\theta,t)\cap S_n\cap Q(\bm\lambda^{(i)}\kappa) | \bigm| Q_{n,t} = Q(\bm\lambda) \bigr] \\
    &=  \E \bigl[ L^{i+1}|\ell(\theta,t)\cap S_{i+1}\cap Q(\bm\lambda^{(i)}\kappa)| \bigm| Q_{n,t} = Q(\bm\lambda) \bigr]\,.
\end{align*}
Splitting into the cases $Q(\bm\lambda^{(i)}\kappa)\in\mathcal{S}_{i+1}$ and $Q(\bm\lambda^{(i)}\kappa)\notin\mathcal{S}_{i+1}$, we can further condition (in both sides) on $Q(\bm\lambda^{(i)}\kappa)\in\mathcal{S}_{i+1}$. Let $\ell'=h_{Q(\bm\lambda^{(i)}\kappa)}^{-1}(\ell(\theta,t))$. After this conditioning,  $L^{i+1}|\ell(\theta,t)\cap S_{i+1}\cap Q(\bm\lambda^{(i)}\kappa)|$ becomes $|\ell'\cap S_0|$ (deterministically) while, by the scale invariance and the independence of the process, the random variable $L^n |\ell(\theta,t)\cap S_n\cap Q(\bm\lambda^{(i)}\kappa) |$ has the same distribution as $L^{n-i-1} |\ell' \cap S_{n-i-1}|$, which in turn has expectation $|\ell'\cap S_0|$. This yields the claim.

The lemma follows from \eqref{eq:def_X'} and \eqref{eq:equal-conditional-expectations}.
\end{proof}

The following is an immediate consequence of \Cref{lemma:decomp X_ni}.

\begin{lemma}\label{lem:unif_ub_exp}
    For each $\bm\lambda\in\Lambda^n$, 
    \begin{align}
    \label{eq:cond exp upper bound}
        \E\bigl[L^n |\ell(\theta,t)\cap S_n|\bigm| Q_{n,t}=Q(\bm\lambda)\bigr]=O(n)\,.
    \end{align}
    Moreover, $n\P\bigl[\ell(\theta,t)\cap S_n\neq\varnothing\bigr]\gtrsim|\ell(\theta,t)\cap[0,1)^2|$.
\end{lemma}

\begin{proof}
    The first claim follows from  \eqref{eq:reduct-to-average step1} noting that $X_{n,i}\le C$, where $C$ is independent of $Q$, $n$, $\theta$, $t$. The second claim follows from the first and \eqref{eq:length-w-square}.
\end{proof}

\Cref{lem:unif_ub_exp} applies to all $\bm\lambda \in \Lambda^n$. By removing the two exceptional sets mentioned in \Cref{sec:exceptional sets}, we can obtain significantly better estimates. The following proposition is the main step towards the proof of \Cref{thm:limit}.

\begin{proposition}\label{prop:expectation_conditional_on_approx_distr}
If $1 \leq k \leq k' \leq n$ and
$\bm\lambda \in \Lambda^n\setminus \bigcup_{i=0}^{n-k'} \cH_{n,i}$ is $(k,\epsilon)$-approximately uniform, then 
    \begin{align}\label{eq:expectation_conditional_on_approx_distr}
        \left| \frac{1}{n} \E\bigl[ L^n |\ell(\theta,t) \cap S_n| \bigm| Q_{n,t} = Q(\bm\lambda) \bigr] -\frac{V(\theta)}{2} \right|
      =
      O\left(\frac{k'}{n} + \frac{\log n}{n}
        +
        \frac{1}{L^k}
        +
        \epsilon
        \right) \,.
    \end{align}
\end{proposition}

\begin{proof}
Throughout this proof, we fix $\bm\lambda, \theta, t$ such that $\ell(\theta,t)\cap Q(\bm\lambda)\neq\varnothing$.

First, by combining \eqref{eq:reduct-to-average step1} with the easy bounds $L^n |\ell(\theta,t) \cap Q(\bm\lambda)|=O(1)$ and $X_{n,i}=O(1)$, we get
\begin{align}\label{eq:reduct-to-average}
\frac{1}{n}
\E\bigl[L^n |\ell(\theta,t) \cap S_n| \bigm| Q_{n,t} = Q(\bm\lambda) \bigr]
=
\frac{1}{n}
\sum_{i=0}^{n-k'} \E \bigl[X_{n,i} \bigm| Q_{n,t} = Q(\bm\lambda)  \bigr]
+
O\left(\frac{k'}{n}\right)
\,.
\end{align}

Fix $0\le i\le n-k'$. Then 
\begin{align*}
    &\E\bigl[X_{n,i} \bigm| Q_{n,t} = Q(\bm\lambda)\bigr]
    \\
    &=
    L^{i+1} \sum_{\kappa\in\Lambda, \kappa \succ \lambda_{i+1}}
    |\ell(\theta,t) \cap Q(\bm\lambda^{(i)}\kappa)| \, \P[Q(\bm\lambda^{(i)}\kappa) \in \cS_{i+1} \mid Q_{n,t} = Q(\bm\lambda)]
    \\
    &=L\sum_{\kappa\in\Lambda, \kappa \succ \lambda_{i+1}}
    \left|\ell(\theta,t_i)\cap Q(\kappa)\right|
    \P[Q(\kappa)\in \cS_{1}\mid Q_{n-i,t_i} = Q(\lambda_{i+1}, \ldots, \lambda_n)]\,,
\end{align*}
where in the last line, we applied the homothety $h^{-1}_{Q(\bm\lambda^{(i)})} : Q(\bm\lambda^{(i)}) \to [0,1)^2$ and we let $t_i = s(t,\bm\lambda,i)$; recall from \eqref{eq:def s(t,lambda,i)} that this means $\ell(\theta,t_i) = h^{-1}_{Q(\bm\lambda^{(i)})}(\ell(\theta,t))$. 
By \Cref{lemma:probability given first survivor}, and recalling that $\bm\lambda\notin\cH_{n,i}$ by assumption, we get
\[
    \left| 
    \frac{\P[Q(\kappa)\in \cS_{1}\mid Q_{n-i,t_i} = Q(\lambda_{i+1}, \ldots, \lambda_n)]}
    {\P[Q(\kappa)\in \cS_{1}\mid Q(\lambda_{i+1}) \in \cS_1] }
    -1
    \right| \lesssim\frac{1}{n-i}\,.
\]
This implies
\begin{align*}
    &\E\bigl[X_{n,i} \bigm| Q_{n,t} = Q(\bm\lambda)\bigr]
    \\
    &=L\sum_{\kappa\in\Lambda, \kappa \succ \lambda_{i+1}}
    \left|\ell(\theta,t_i)\cap Q(\kappa)\right|
    \P[Q(\kappa)\in \cS_{1}\mid Q(\lambda_{i+1}) \in \cS_1] +O\left(\frac{1}{n-i}\right)
    \\
    &\stackrel{\eqref{eq:squarewise-uniform}}{=} L^2 \sum_{\kappa\in\Lambda, \kappa \succ \lambda_{i+1}}
    \left|\ell(\theta,t_i)\cap Q(\kappa)\right|
    \P[Q(\kappa) \in \cS_1 \land Q(\lambda_{i+1}) \in \cS_1] +O\left(\frac{1}{n-i}\right)\,
    \\
    &\stackrel{\eqref{eq:def V(theta,eta,kappa) and E}}{=}
    L^2 \sum_{\kappa\in\Lambda, \kappa \succ \lambda_{i+1}}
    V(\theta, (\lambda_{i+1}, \ldots, \lambda_{i+k}), \kappa)
    \P[Q(\kappa) \in \cS_1 \land Q(\lambda_{i+1}) \in \cS_1]
    \\
    &\qquad+
    O\left(\frac{1}{n-i}
    +
    \sum_{\kappa \in \Lambda} \Err(\theta, (\lambda_{i+1}, \ldots, \lambda_{i+k}), \kappa)
    \right)\,.
\end{align*}
where in the last line, we used the observation that $\ell(\theta,t_i) \cap Q(\lambda_{i+1}, \ldots, \lambda_{i+k}) =  h^{-1}_{Q(\bm\lambda^{(i)})}(\ell(\theta,t) \cap  Q(\bm\lambda)) \neq \varnothing$.

Summing over $0 \leq i \leq n-k'$ and dividing by $n$,
\begin{align}\label{eq:reduction_to_V sum}
\begin{split}
    &\frac{1}{n} \sum_{i=0}^{n-k'} \E\bigl[X_{n,i} \bigm| Q_{n,t} = Q(\bm\lambda)\bigr]\\
    &=
    \frac{L^{2}}{n}  \sum_{i=0}^{n-k'}\sum_{\kappa\in\Lambda, \kappa \succ \lambda_{i+1}}
    V(\theta, (\lambda_{i+1}, \ldots, \lambda_{i+k}), \kappa)
    \P[Q(\kappa) \in \cS_1 \land Q(\lambda_{i+1}) \in \cS_1]
    \\
    &\qquad+
    O\left(\frac{\log n}{n}
    +
    \frac{1}{n}\sum_{i=0}^{n-k} \sum_{\kappa \in \Lambda}  \Err(\theta, (\lambda_{i+1}, \ldots, \lambda_{i+k}), \kappa)
    \right)\,.
\end{split}
\end{align}
Since $\bm\lambda$ is $(k,\epsilon)$-approximately uniform, for all $\bm\eta \in \Lambda^k$,
\begin{align}
\label{eq:approx uniform k}
\begin{split}    
    \Count(\bm\lambda^{(n-k'+k)},\bm\eta)
    &=
    \#\{i\in \{0,\ldots, n-k'\} \,:\,(\lambda_{i+1},\ldots,\lambda_{i+k})= \bm\eta\}
    \\
    &\stackrel{\eqref{eq:approx-uniform}}{=}
    \frac{n}{L^{2k}}  + O\left(\frac{\epsilon n}{L^{2k}} + \frac{k'}{L^{2k}}\right)  \,.
\end{split}
\end{align}
Thus, we have 
\begin{align}\label{eq:V_loc_to_glob}
\begin{split}
    &
    \frac{L^{2}}{n} \sum_{i=0}^{n-k'}
    \sum_{\kappa\in\Lambda, \kappa \succ \lambda_{i+1}}
        V(\theta, (\lambda_{i+1}, \ldots, \lambda_{i+k}), \kappa)
        \P[Q(\kappa) \in \cS_1 \land Q(\lambda_{i+1}) \in \cS_1]
    \\
    &=
    \frac{L^{2}}{n}\sum_{\bm\eta \in \Lambda^k}
    \sum_{\kappa\in\Lambda, \kappa \succ \eta_1}
    V(\theta,\bm\eta,\kappa)
    \P[Q(\kappa) \in \cS_1 \land Q(\eta_1) \in \cS_1]
    \cdot  
    \Count(\bm\lambda^{(n-k'+k)},\bm\eta)
    \\
    &\stackrel{\eqref{eq:approx uniform k}}{=}
    L^{2-2k}\sum_{\eta_1, \kappa \in \Lambda, \kappa \succ \eta_1}
    \P[Q(\kappa) \in \cS_1 \land Q(\eta_1) \in \cS_1]
    \sum_{(\eta_2,\ldots,\eta_k) \in \Lambda^{k-1}}
    V(\theta,\bm\eta,\kappa)
    +
    O(\epsilon + \tfrac{k'}{n})\,.
\end{split}
\end{align}
where in the last line we used $V(\theta,\bm\eta,\kappa) = O(1)$.

Since for all $\eta_1\in\Lambda$,
\begin{align*}
    \sum_{(\eta_2,\ldots,\eta_k) \in \Lambda^{k-1}}
    V(\theta,\bm\eta,\kappa)
    &=
    L^{2k}\sum_{(\eta_2,\ldots,\eta_k) \in \Lambda^{k-1}}
    \int_{\R}  |\ell(\theta, s) \cap Q(\bm\eta)| |\ell(\theta, s) \cap Q(\kappa)| \, ds
    \\
    &=
    L^{2k}\int_{\R}  |\ell(\theta, s) \cap Q(\eta_1)| |\ell(\theta, s) \cap Q(\kappa)| \, ds\,,
\end{align*}
we arrive at
\begin{align}\label{eq:combine to V}
\begin{split}
    &L^{2-2k}\sum_{\eta_1, \kappa \in \Lambda, \kappa \succ \eta_1}
    \P[Q(\kappa) \in \cS_1 \land Q(\eta_1) \in \cS_1]
    \sum_{(\eta_2,\ldots,\eta_k) \in \Lambda^{k-1}}
    V(\theta,\bm\eta,\kappa)
    \\
    &=
    L^2\sum_{\eta_1, \kappa \in \Lambda, \kappa \succ \eta_1}
    \P[Q(\kappa) \in \cS_1 \land Q(\eta_1) \in \cS_1]
    \int_{\R}  |\ell(\theta, s) \cap Q(\eta_1)| |\ell(\theta, s) \cap Q(\kappa)| \, ds
    \\
    &\stackrel{\eqref{eq:V(theta) alternative}}{=}
    \frac{V(\theta)}{2}\,.
\end{split}
\end{align}
Finally, we use the fact that $\bm\lambda$ is $(k,\eps)$-approximately uniform, together with \Cref{lem:length_to_integral}, to bound the second error term in \eqref{eq:reduction_to_V sum}:
\begin{align}
\label{eq:bound error term sum i sum kappa}
    \frac{1}{n}\sum_{i=0}^{n-k'} \sum_{\kappa \in \Lambda} \Err(\theta, (\lambda_{i+1}, \ldots, \lambda_{i+k}), \kappa)
    \stackrel{\eqref{eq:approx uniform k}}{\lesssim}
    \frac{1}{L^{2k}}
    \sum_{\bm\eta \in \Lambda^k}
    \sum_{\kappa\in\Lambda} \Err(\theta, \bm\eta, \kappa)
    \stackrel{\eqref{eq:length-to-integral-sum}}{\lesssim}
    \frac{1}{L^k} \,.
\end{align}

  Combining \eqref{eq:reduct-to-average}--\eqref{eq:bound error term sum i sum kappa} yields \eqref{eq:expectation_conditional_on_approx_distr}.
\end{proof}

\begin{proof}[Proof of Theorem \ref{thm:limit}]
Let us fix $\theta\in [0,\pi]$ and $0<\varepsilon<\tfrac14$. Let $k = \lceil \log_L (1/\eps) \rceil$. For $n\ge \frac{k}{\varepsilon^2}$, let $k' = \lfloor \sqrt{n} \rfloor$ and  
\[
    \cB_n
    =
    \bigcup_{\bm\eta\in\Lambda^k}
    \left\{\bm\lambda \in \Lambda^n \,:\, \left|\Count(\bm\lambda, \bm\eta)-\frac{n-k+1}{L^{2k}}\right|>\frac{\varepsilon (n-k+1)}{L^{2k}}\right\}
    \cup
    \bigcup_{i=0}^{n-k'} \cH_{n,i}
    \,,
\]
where $\mathcal{H}_{n,i}$ is as in Lemma \ref{lem:H_n_size}. By Lemmas \ref{lem:chern_k_blocks} and \ref{lem:H_n_size}, 
\begin{align}
    \label{eq:B_n/Lambda^n}
    \frac{\#\cB_n}{\#\Lambda^n}
    =
    O_k(1)\exp(-\Omega_{k,\varepsilon}(n)) 
    +
    O(1)\exp(-\Omega(k'))
    =
    O_\eps(\exp(-\Omega_{\eps}(\sqrt{n})))
    \,.
\end{align}

Let $p_n(\theta,t) = \P[\ell(\theta,t)\cap S_n\neq\varnothing]$. The event $\ell(\theta,t)\cap S_n\neq\varnothing$ is the disjoint union of the events $Q_{n,t} \in \cB_n$ and $Q_{n,t} \in \cD_n \setminus \cB_n$. Thus, we have
\begin{align*}
    &|\ell(\theta,t) \cap [0,1)^2|
    \\&=
    \E\bigl[ L^n |\ell(\theta,t) \cap S_n|\bigr]
    \\
    &=
    \E\bigl[ L^n |\ell(\theta,t) \cap S_n| \bigm| Q_{n,t} \in \cD_n \setminus \cB_n \bigr] \cdot (p_n(\theta,t) - \P[Q_{n,t} \in \cB_n])
    \\
    &\qquad+
    \E\bigl[ L^n |\ell(\theta,t) \cap S_n| \bigm| Q_{n,t} \in \cB_n \bigr] \cdot \P[Q_{n,t} \in \cB_n]
    \\
    &\stackrel{\eqref{eq:cond exp upper bound}}{=}
    \E\bigl[ L^n |\ell(\theta,t) \cap S_n| \bigm| Q_{n,t} \in \cD_n \setminus \cB_n \bigr] \cdot p_n(\theta,t) + O(n) \cdot \P[Q_{n,t} \in \cB_n]
    \\
    &\stackrel{\eqref{eq:expectation_conditional_on_approx_distr}}{=}
    \left(
    \frac{V(\theta)}{2}
    +
    O(\epsilon)
    \right)
    \cdot n p_n(\theta,t) + O(n) \cdot \P[Q_{n,t} \in \cB_n]
    \,.
\end{align*}

Integrating in $t$ and using \Cref{lemma:integral Qnt in B} and \eqref{eq:B_n/Lambda^n} gives
\[
    1
    =
    \left(
    \frac{V(\theta)}{2}
    +
    O(\epsilon)
    \right)
    \cdot n \E|\proj_\theta(S_n)|
    +
    O_\eps(\exp(-\Omega_{\eps}(\sqrt{n})))
    \,.
\]
Recall \Cref{lem:V_lb}: $\inf_\theta V(\theta) > 0$. Thus, if $\eps$ is small enough, then
\[
    n \E|\proj_\theta(S_n)|
    =
    \frac{2}{V(\theta)}
    +
    O(\epsilon)
    +
    O_\eps(\exp(-\Omega_{\eps}(\sqrt{n})))
    \,.
\]
The theorem follows by letting $n\to\infty$ and $\varepsilon\to0$.
\end{proof}

\appendix

\section{The Bateman--Volberg lower bound}
\label{sec:bateman-volberg}

Fix $L \geq 2$, and let $\cD_n$ be the set of squares obtained by dividing the unit square into an $L^n \times L^n$ grid.

\begin{theorem}
    \label{thm:bateman-volberg analogue}
    Let $[0,1)^2 = S_0 \supset S_1 \supset S_2 \supset \cdots$ be any sequence of sets such that for each $n$, $S_n$ is the union of $L^n$ squares in $\cD_n$, one from each column. Then 
    \begin{align*}
        \Fav(S_n) \gtrsim_L \frac{\log n}{n} \,.
    \end{align*}
\end{theorem}

To prove \Cref{thm:bateman-volberg analogue}, we adapt the argument by Bateman and Volberg \cite{BatemanVolberg10}. We begin with a counting lemma. Given $Q \in \cD_n$, let $(x(Q), y(Q))$ denote the lower-left corner of $Q$. Let $\cS_n = \{Q \in \cD_n : Q \subset S_n\}$.

\begin{lemma}[Counting pairs]
\label{lemma:one per col counting}
    For any $i,j \in [n]$,
    \begin{align}
    \label{eq:one per col counting}
    \begin{aligned}
    &\#\{(Q,Q') \in \cS_n \times \cS_n:
        |x(Q) - x(Q')| \leq L^{-i}
        ,
        |y(Q) - y(Q')| \geq L^{-j}
    \}
        \\
        &\lesssim_L L^{2n+j-2i}
    \end{aligned}
    \end{align}
\end{lemma}

\begin{proof}
    The one-per-column rule implies the following: Let $Q, Q' \in \cS_n$. If $x(Q), x(Q') \in [\frac{a}{L^j}, \frac{a+1}{L^j})$ for some $j \in \{1, \ldots, n\}$ and some $a \in \{0, \ldots, L^j-1\}$, then $y(Q), y(Q') \in [\frac{b}{L^j}, \frac{b+1}{L^j})$ for some $b \in \{0, \ldots, L^j-1\}$.

    As a result, if $(Q,Q')$ is in the set defined in \eqref{eq:one per col counting}, then 
    \begin{align}
        \label{eq:x,x' nbhd}
        \text{there exists $a \in \{1, \ldots, L^j-1\}$ such that $x(Q), x(Q') \in B(a L^{-j}, L^{-i})$}.
    \end{align}
    By the one-per-column rule again, the number of pairs $(Q,Q') \in \cS_n \times \cS_n$ satisfying \eqref{eq:x,x' nbhd} is $\lesssim_L L^j (L^{n-i})^2$.
\end{proof}

Next, we prove a key $L^2$ estimate.

\begin{lemma}[$L^2$ estimate]
    \label{lemma:L^2 estimate Theta_k}
    Define 
    \begin{align*}
        f_{n,\theta}(t) &= \# \{ Q \in \cS_n : \ell(\theta,t) \cap Q \neq \varnothing\} = \sum_{Q \in \cS_n} \bbone_{\proj_\theta(Q)}(t)
        \\
        \Theta_k &= \{\theta \in [0,\pi/2] : \tan\theta \in [L^k, L^{k+1}]\} \,.
    \end{align*} 
    Then
    \begin{align}
        \label{eq:L^2 estimate Theta_k}
        \int_{\Theta_k} \| f_{n,\theta} \|_2^2 \, d\theta \lesssim \frac{n}{L^{2k}}, \qquad \text{for all } k = 1, \ldots, \lfloor \log_L n \rfloor \,.
    \end{align}
\end{lemma}

\begin{proof}
    Fix $k \in \{1, \ldots, \lfloor \log_L n \rfloor\}$ and define
    \begin{align*}
        \cP_k = \{(Q,Q') \in \cS_n \times \cS_n : \exists \theta \in \Theta_k \text{ such that } \proj_\theta(Q) \cap \proj_\theta(Q') \neq \varnothing\} \, .
    \end{align*}
    By Fubini,
    \begin{align}
        \label{eq:fubini sum pairs}
        \int_{\Theta_k} \| f_{n,\theta} \|_2^2 \, d\theta 
        &= \sum_{(Q, Q') \in \cP_k} \int_{\Theta_k} |\proj_\theta(Q) \cap \proj_\theta(Q')| \, d\theta \,.
    \end{align}
    For any pair $(Q,Q') \in \cD_n \times \cD_n$, the integrand on the right-hand side of \eqref{eq:fubini sum pairs} is $O(L^{-n})$, so
    \begin{align}
    \label{eq:int Theta_k upper bound}
    \begin{aligned}    
        \int_{\Theta_k} |\proj_\theta(Q) \cap \proj_\theta(Q')| \, d\theta 
        &\lesssim 
        L^{-n} |\{\theta \in \Theta_k : \proj_\theta(Q) \cap \proj_\theta(Q') \neq \varnothing\}| 
        \\
        &\lesssim
        L^{-n} \min\left(L^{-k}, \frac{L^{-n}}{d(Q,Q')} \right)
        \,.
    \end{aligned}
    \end{align}

    Next, we claim that
    \begin{align}
        \label{eq:P_k decompose}
        \cP_k \subset \cA \cup \bigcup_{j=1}^{n-k} \cP_{k,j}
    \end{align}
    where
    \begin{align*}
        \cA &= \{(Q, Q') \in \cS_n \times \cS_n : |x(Q) - x(Q')| \leq L^{-n}\}
        \\
        \cP_{k,j} &= \{(Q,Q') \in \cS_n \times \cS_n : |x(Q) - x(Q')| \sim L^{-j-k}, |y(Q) - y(Q')| \sim L^{-j}\} \,
    \end{align*}
    and where the implied constants in the definition of $\cP_{k,j}$ are chosen appropriately, as follows: if $(Q,Q') \in \cP_k \setminus \cA$, then $|x(Q) - x(Q')| \geq 2L^{-n}$; consequently, one can check that there exist constants $c_1, c_2$ (depending only on $L$) such that
    \begin{align*}
        c_1 L^k \leq \frac{|y(Q) - y(Q')|}{|x(Q) - x(Q')|} \leq c_2 L^k \qquad\text{for all } (Q,Q') \in \cP_k \setminus \cA\,. 
    \end{align*}
    It follows that by choosing the implied constants appropriately in the definition of $\cP_{k,j}$, we have $\cP_k \setminus \cA \subset \bigcup_{j=0}^{n-k} \cP_{k,j}$. This proves \eqref{eq:P_k decompose}.

    The one-per-column rule implies that $\# \cA \lesssim L^n$, so
    \begin{align*}
        \sum_{(Q, Q') \in \cA} 
        \int_{\Theta_k} |\proj_\theta(Q) \cap \proj_\theta(Q')| \, d\theta
        \stackrel{\eqref{eq:int Theta_k upper bound}}{\lesssim}
        \# \cA \cdot L^{-n} L^{-k} 
        \lesssim
        L^{-k}
        \,.
    \end{align*}
    For each $j \in \{1, \ldots, n-k\}$,
    \begin{align*}
         \sum_{(Q, Q') \in \cP_{k,j}}
         \int_{\Theta_k} |\proj_\theta(Q) \cap \proj_\theta(Q')| \, d\theta
         \stackrel{\eqref{eq:int Theta_k upper bound}}{\lesssim}
         \# \cP_{k,j} \cdot \frac{(L^{-n})^2}{L^{-j}}
         \stackrel{\eqref{eq:one per col counting}}{\lesssim}
         L^{-2k} \,.
    \end{align*}
    By combining these two estimates with \eqref{eq:fubini sum pairs} and \eqref{eq:P_k decompose}, we obtain
    \begin{align*}
        \int_{\Theta_k} \| f_{n,\theta} \|_2^2 \, d\theta 
        &\lesssim
        L^{-k} + (n-k) L^{-2k} \,,
    \end{align*}
    which proves the lemma.
\end{proof}

\begin{proof}[Proof of \Cref{thm:bateman-volberg analogue}]
    Let $f_{n,\theta}$ and $\Theta_k$ be as in the statement of \Cref{lemma:L^2 estimate Theta_k}. Let $k \in \{1, \ldots, \lfloor \log_L n \rfloor\}$. By applying Cauchy--Schwarz twice and applying \Cref{lemma:L^2 estimate Theta_k},
    \begin{align*}
        \frac{1}{L^{2k}} \sim |\Theta_k|^2 \sim \left(\int_{\Theta_k} \|f_{n,\theta}\|_1 \, d\theta\right)^2
        &\leq
        \left(\int_{\Theta_k} |\proj_\theta S_n|^{1/2} \|f_{n,\theta}\|_2  \, d\theta\right)^2
        \\
        &\leq
        \int_{\Theta_k} |\proj_\theta S_n| \,d\theta \int_{\Theta_k} \|f_{n,\theta}\|_2^2  \, d\theta
        \\
        &\stackrel{\eqref{eq:L^2 estimate Theta_k}}{\lesssim_L}
        \int_{\Theta_k} |\proj_\theta S_n| \,d\theta \cdot \frac{n}{L^{2k}}
    \end{align*}
    Then
    \begin{align*}
        \Fav(S_n)
        \geq
        \sum_{k=1}^{\lfloor \log_L n\rfloor} 
        \int_{\Theta_k} |\proj_\theta S_n| \,d\theta
        \gtrsim_L
        \frac{\log n}{n} \,.
    \end{align*}
\end{proof}

\begin{remark}
    The proof of \Cref{lemma:L^2 estimate Theta_k} above can also be rephrased in terms of conical Riesz energies, as introduced in \cite{ChangTolsa20} (see (1.8) in that paper). Let $\mu_n$ be the Lebesgue measure on $S_n$, normalized to be a probability measure, and let $\proj_\theta \mu_n$ denote the pushforward of $\mu_n$ by $\proj_\theta$. Then by applying \cite[Proposition 3.2]{ChangTolsa20} to smooth approximations of $\mu_n$, we can obtain
    \[
        \int_{\Theta_k}
        \|\proj_\theta \mu_n\|_2^2
        \, d\theta
        =
        \iint_{x-y \in C_k} \, \frac{d \mu_n(x) \, d\mu_n(y)}{|x-y|}
        ,
    \]
    where $C_k$ is the cone $\bigcup_{\theta \in \Theta_k} \proj_\theta^{-1}(\{0\})$. The right-hand side can then be bounded using \Cref{lemma:one per col counting}.
\end{remark}

\section{Classical concentration inequalities}
\label{sec:elementaryprob}

We collect some classical concentration inequalities.

\begin{lemma}[Hoeffding's inequality]
    \label{lemma:hoeffding}
    Let $X_1, \ldots, X_n$ be independent bounded random variables such that $a_i \leq X_i \leq b_i$. Then, with $X = \sum_{i = 1}^{n} X_i $,
    \begin{align*}
        \P\bigl[|X-\E[X]| \geq \epsilon\E[X]\bigr]
        \leq
        2\exp\left(- \frac{2\epsilon^2}{\sum_{i=1}^n (b_i-a_i)^2}\right)
        \qquad\text{for all } \epsilon \in [0,1] \,.
    \end{align*}
\end{lemma}

\begin{proof}
    See, e.g., \cite[Theorem 2.8]{BoucheronLugosiMassart13}.
\end{proof}

\begin{lemma}[Chernoff's inequality]
    \label{lemma:chernoff}
    Let $X_1, \ldots, X_n$ be independent $\{0,1\}$-valued random variables. Then, with $X = \sum_{i = 1}^{n} X_i $,
    \begin{align*}
        \P\bigl[|X-\E[X]| \geq \epsilon\E[X]\bigr]
        \leq
        2\exp(- \tfrac{1}{4} \epsilon^2 \E[X])
        \qquad\text{for all } \epsilon \in [0,1] \,.
    \end{align*}
\end{lemma}

\begin{proof}
    See, e.g., \cite[Corollary 1.9]{TaoVu06}.
\end{proof}

\begin{lemma}[Chernoff lower tail bound for negatively correlated events]
    \label{lemma:chernoff neg dep}
    Let $X_1, \ldots, X_n$ be $\{0,1\}$-valued random variables that are ``jointly negatively correlated'' in the sense that 
    \begin{align}
        \label{eq:neg dep}
        \P[\forall i \in I, X_i = 0]
        \leq
        \prod_{i \in I}
        \P[X_i = 0]
        \qquad \text{for all } I \subset [n].
    \end{align}
    Then, with $X = \sum_{i = 1}^{n} X_i $,
    \begin{align*}
        \P\bigl[X\leq (1 - \epsilon)\E[X]\bigr]
        \leq
        \exp(- \tfrac{1}{2} \epsilon^2 \E[X])
        \qquad\text{for all } \epsilon \in [0,1] \,.
    \end{align*}
\end{lemma}

\begin{proof}
The proof of Chernoff's inequality for independent random variables is still valid when independence is weakened to \eqref{eq:neg dep}. See, e.g., \cite[Theorem 1.10.24]{Doerr20} or \cite[Theorem 3.4]{PanconesiSrinivasan97}. 
\end{proof}

\end{document}